\pdfoutput=1
\RequirePackage{ifpdf}
\ifpdf 
\documentclass[pdftex]{sigma}
\else
\documentclass{sigma}
\fi

\numberwithin{equation}{section}

\newtheorem{Theorem}{Theorem}[section]
\newtheorem{Lemma}[Theorem]{Lemma}
\newtheorem{Proposition}[Theorem]{Proposition}

\def\half{{\textstyle\frac12}}
\def\tf#1#2{{\textstyle\frac{#1}{#2}}}
\def\mapto#1{\mathop{\longleftarrow}\limits^{#1}}

\begin{document}

\allowdisplaybreaks

\renewcommand{\PaperNumber}{095}

\FirstPageHeading

\ShortArticleName{The Chazy XII Equation and Schwarz Triangle Functions}

\ArticleName{The Chazy XII Equation\\ and Schwarz Triangle Functions}

\Author{Oksana BIHUN and Sarbarish CHAKRAVARTY}
\AuthorNameForHeading{O.~Bihun and S.~Chakravarty}
\Address{Department of Mathematics, University of Colorado, Colorado Springs, CO 80918, USA}
\Email{\href{mailto:obihun@uccs.edu}{obihun@uccs.edu}, \href{mailto:schakrav@uccs.edu}{schakrav@uccs.edu}}

\ArticleDates{Received June 21, 2017, in f\/inal form December 12, 2017; Published online December 25, 2017}

\Abstract{Dubrovin~[\textit{Lecture Notes in Math.}, Vol.~1620, Springer, Berlin, 1996, 120--348]
showed that the Chazy XII equation $y'''- 2yy''+3y'^2 = K(6y'-y^2)^2$, $K \in \mathbb{C}$,
is equivalent to a projective-invariant equation for an af\/f\/ine connection
on a one-dimensional complex manifold with projective structure.
By exploiting this geometric connection it is shown that the Chazy~XII
solution, for certain values of~$K$,
can be expressed as $y=a_1w_1+a_2w_2+a_3w_3$
where $w_i$ solve the generalized Darboux--Halphen system. 
This relationship holds only for certain values
of the coef\/f\/icients $(a_1,a_2,a_3)$ and the Darboux--Halphen parameters
$(\alpha, \beta, \gamma)$, which are enumerated in Table~\ref{t:gdh}.
Consequently, the Chazy XII solution $y(z)$ is parametrized by a particular
class of Schwarz triangle functions $S(\alpha, \beta, \gamma; z)$
which are used to represent the solutions~$w_i$ of the Darboux--Halphen
system. The paper only considers the case where $\alpha+\beta+\gamma<1$.
The associated triangle functions are related among themselves
via rational maps that are derived from
the classical algebraic transformations of hypergeometric functions.
The Chazy XII equation is also shown to be equivalent to a Ramanujan-type
dif\/ferential system for a triple $(\hat{P}, \hat{Q},\hat{R})$.}


\Classification{34M45; 34M55; 33C05}

\section{Introduction}\label{section1}
In 1911, J.~Chazy~\cite{Chazy3} considered the classif\/ication problem of all third order dif\/ferential equations of the form $y''' = F(z, y, y', y'')$ possessing the Painlev\'e property, where the prime $'$ denotes ${\rm d}/{\rm d}z$, and $F$ is a polynomial in $y$, $y'$, $y''$ and locally analytic in~$z$. A~dif\/ferential equation in the complex plane is said to have the Painlev\'e property if its general solution has no movable branch points. In his work, Chazy introduced thirteen classes of equations referred to as Chazy classes I--XIII. Among these, classes III and XII are particularly
interesting as their general solutions possess a movable barrier, i.e., a closed curve in the complex plane across which the solutions can not be analytically continued. That is, the solutions are analytic (or meromorphic) on either side of the barrier depending on the prescribed initial conditions. Both the Chazy III and XII equations can be expressed together as
\begin{gather}
y'''- 2yy''+3y'^2 = K\big(6y'-y^2\big)^2, \qquad K=0 \qquad \mbox{or} \qquad K=\frac{4}{36-k^2} .\label{gch}
\end{gather}
The standard form for the Chazy XII equation is given by~\eqref{gch} with the parameter $K=4/(36-k^2)$, whereas the Chazy~III equation
corresponds to $K=0$, and is the limiting case of Chazy~XII
as $k \to \infty$. Chazy~\cite{Chazy3} observed the remarkable fact that~\eqref{gch}
is linearizable via the hypergeometric equation
\begin{gather*}s(s-1)\chi''+ \big(\tf{7s}{6} -\half\big)\chi'+\big(\tf{1}{4k^2}-\tf{1}{144}\big)\chi=0,\end{gather*}
for $k \in {\mathbb C} \cup \infty$, $k\neq 0$. He also showed that \eqref{gch}
possesses the Painlev\'e property when $K=0$ or when the value
of the parameter $k$ is a positive integer such that $k > 1$
and $k \neq 6$ (see also~\cite{Cosgrove}). The $k=0$ case is linearizable
by Airy's equation $\chi''=cs\chi$, $c$ constant~\cite{CO1996}, and the limiting cases
of $k=\pm 6$ can be solved via elliptic functions~\cite{Chazy3, Cosgrove}.
For $K=0$ or for integer values $k>6$, the general solution of~\eqref{gch} possesses a movable natural barrier in the complex plane.

The Chazy III equation arises in
mathematical physics in studies concerning magnetic mono\-po\-les~\cite{Atiyah-H},
self-dual Yang--Mills and Einstein equations \cite{CAC90, Hitchin}, topological
f\/ield theory~\cite{Dubrovin}, as well as special reductions of hydrodynamic type
equations \cite{Ferapontov} and incompressible f\/luids~\cite{Rosenhead}.
The Chazy XII equation is related to the generalized
Darboux--Halphen system~\cite{Ablowitz-H1,Ablowitz-H2}, which arises
in reductions of self-dual Yang--Mills equations associated with
the gauge group of dif\/feomorphism $\operatorname{Dif\/f}(S^3)$ of a 3-sphere~\cite{CAT92}
as well as ${\rm SU}(2)$-invariant hypercomplex manifolds~\cite{Hitchin1}
with self-dual Weyl curvature~\cite{CA96}.
More recently, the Chazy~XII equation with specif\/ic values of the parameter~$k$
has been linked to the study of vanishing Cartan curvature invariant for
certain types of rank~2 distributions on a 5-manifold~\cite{Randall16, Randall17}.
Consequently, there has been renewed interest in the study of~\eqref{gch} and
the singularity structure of its solutions in the complex plane. Interested
readers are referred to the comprehensive works by Bureau~\cite{Bureau1},
Clarkson and Olver~\cite{CO1996} and Cosgrove~\cite{Cosgrove}.

A signif\/icant aspect of \eqref{gch} is the fact that its general solutions
can be expressed in terms of Schwarz triangle functions which
def\/ine conformal mappings from the upper
half complex plane onto a region bounded by three circular arcs
(see, e.g.,~\cite{Nehari}). For $K=0$,
solutions of the Chazy III equation are related to the automorphic forms
associated with the modular group~${\rm SL}_2(\mathbb{Z})$ and its subgroups~\cite{CA09}.
This fact can be traced back to the work of S.~Ramanujan. In 1916,
Ramanujan~\cite{Ramanujan-arith}, \cite[pp. 136--162]{Ramanujan-collect}
introduced certain functions $P(q)$, $Q(q)$, $R(q)$, $q := e^{2 \pi iz}$, $\operatorname{Im}(z)>0$, which correspond to the
(f\/irst three) Eisenstein series associated with the modular group ${\rm SL}_2(\mathbb{Z})$.
He showed that these functions satisfy the dif\/ferential relations
\begin{gather}
\frac{1}{2 \pi i} P'(z) =\frac{P^2-Q}{12}, \qquad
\frac{1}{2 \pi i}Q'(z)=\frac{PQ-R}{3}, \qquad
\frac{1}{2 \pi i} R'(z) =\frac{PR-Q^2}{2}.
\label{dPQR}
\end{gather}
System \eqref{dPQR} can be reduced to a single third order dif\/ferential
equation for $P(q)$. In fact, the function $y(z):= \pi i P(q)$ satisf\/ies
the Chazy III equation.

For $K \neq 0$, the solutions of the Chazy XII equation are also related
to Schwarz functions automorphic on curvilinear triangles that tessellate
the interior of the natural barrier
for integer values of the parameter $k > 6$. For $2 \leq k \leq 5$
the solutions of \eqref{gch} are expressed via the
polyhedral functions which are rational functions
associated with symmetry groups of solids (see, e.g.,~\cite{Ford}).

In this paper, we primarily consider the Chazy XII equation \eqref{gch}
for integer values $k > 6$. The main objective of this paper
is to f\/ind all solutions $y(z)$ that are given in terms of Schwarz
triangle functions. The central result to achieve this goal is
the following:

{\it If $y(z) = 6Y(z)$ solves \eqref{gch} with $K \neq 0$, then
$Y(z)$ is a convex linear combination
\begin{gather*} Y(z) = \beta_1w_1 + \alpha_1w_2 + \gamma_1w_3, \qquad
\alpha_1, \beta_1, \gamma_1 > 0, \qquad \alpha_1+\beta_1+\gamma_1=1 \end{gather*}
of the variables $w_1$, $w_2$, $w_3$ satisfying the Darboux--Halphen system~\eqref{gdh} with
parameters $(\alpha{,} \beta{,} \gamma)$. The relationship holds only
for special choices of $(\alpha_1, \beta_1)$ and $(\alpha, \beta, \gamma)$
listed in Table~{\rm \ref{t:gdh}} of Section~{\rm \ref{section3.3}}. The variables $w_1$, $w_2$, $w_3$
are expressed in terms of the Schwarz triangle function $s(z) = S(\alpha, \beta, \gamma; z)$ and its derivatives as shown in~\eqref{wdef}.}

The above result is obtained by extending the geometric formulation
of~\eqref{gch} by B.~Dub\-ro\-vin~\cite{Dubrovin} to recast~\eqref{gch} as
a purely algebraic relation. The Chazy XII solution $y(z)$ is expressed
in a parametric form involving hypergeometric functions $_2F_1$
\begin{gather*} y = y(s), \qquad z=z(s) , \end{gather*}
where $s(z) = S(\alpha, \beta, \gamma; z)$ is the Schwarz triangle function
(see Section~~\ref{section2.2} for details).
Thus, Table~\ref{t:gdh} provides explicit parametrizations
of the Chazy XII solutions in terms of a distinct family of Schwarz
triangle functions. This family includes previously known functions, e.g., $S(\half,\tf13,\tf1k;z)$, $6 < k \in \mathbb{N}$ as well as
some new cases. We believe that
our approach is new, and that it can be applied to study similar
nonlinear equations with natural barrier~\cite{SC13, Maier11}.

The paper is organized as follows. Section~\ref{section2.1} develops necessary
geometric background on af\/f\/ine connections on a one-dimensional
complex manifold with a projective structure fol\-lowing~\cite{Dubrovin}.
In Section~\ref{section2.2}, the generalized Darboux--Halphen system is introduced and its
solutions in terms of Schwarz triangle functions are discussed. The
connection between the Chazy~XII equation and the
generalized Darboux--Halphen
system described in Section~\ref{section3.1} leads to a purely algebraic formulation
of~\eqref{gch}. The solvability conditions of this algebraic system
are analyzed in Section~\ref{section3.2}, and these lead to a classif\/ication of the
Chazy XII solutions according to their pole structures inside
the natural barrier.
In Section~\ref{section3.3}, a Ramanujan-type triple of functions is introduced.
These functions satisfy a system of f\/irst order equations
that is similar to~\eqref{dPQR} and is equivalent to the Chazy~XII equation.
In Section~\ref{section4.1}, the Schwarz triangle function is presented as
the inverse to the conformal map def\/ined by the ratio of two linearly
independent solutions of the hypergeometric equation. The parametrizations
of the Ramanujan-like triple as well as the Chazy XII solution in terms
of hypergeometric functions are also discussed here. Section~\ref{section4.2} outlines
the rational transformations between the distinct Schwarz
functions parameterizing the Chazy XII solution. These rational maps are
derived from the well-known algebraic transformations among the
hypergeometric functions. In order to make the paper self-contained,
we have included two appendices. Appendix~\ref{appendixA} contains a proof of Lemma~\ref{lemma1} introduced
in Section~\ref{section3.2}, while Appendix~\ref{appendixB} contains an elementary derivation of the
radius of the natural barrier for the Chazy XII solution.

\section{Background}\label{section2}
\subsection{Af\/f\/ine connection and projective structure}\label{section2.1}
We begin this section by reviewing the relation between the solution $y(z)$
in~\eqref{gch} with af\/f\/ine connections on a one-dimensional
complex manifold.
We consider dif\/ferential forms of order $m \in \mathbb{N}$ denoted by
$f=f(z){\rm d}z^m$ on a one-dimensional
complex manifold with local coordinate~$z$, where $f(z)$ is a holomorphic
(or meromorphic) function. Under the local change of coordinates $z \to \tilde{z}(z)$,
$f$ transforms according to $f(z)dz^m = \tilde{f}(\tilde{z}){\rm d}\tilde{z}^m$. The covariant
derivative of a $m$-dif\/ferential is a $(m+1)$-dif\/ferential $\nabla f$
def\/ined by
\begin{gather*}
 \nabla f = \nabla f(z) {\rm d}z^{m+1}, \qquad \nabla f(z):= f'(z) - m\eta(z)f(z) ,
\end{gather*}
where $\eta = \eta(z){\rm d}z$ is a holomorphic (or meromorphic) af\/f\/ine connection on the manifold. It follows from the transformation property of $\nabla f$ that~$\eta$ must transform as
\begin{gather}
\tilde{\eta}(\tilde{z}){\rm d}\tilde{z} = \eta(z) {\rm d}z -
\big(\tilde{z}''(z)/\tilde{z}'(z)\big){\rm d}z . \label{eta}
\end{gather}
The ``curvature'' associated with $\eta$ is def\/ined by the quadratic dif\/ferential $\Omega=~\Omega(z)dz^2$ as
\begin{gather*} \Omega(z):= \eta'(z) - \half \eta(z)^2 . \end{gather*}
In fact, $\Omega$ is a {\it projective} connection transforming under local
change of coordinates as
\begin{gather}
\tilde{\Omega}(\tilde{z}){\rm d}\tilde{z}^2 =
\Omega(z){\rm d}z^2 - \{\tilde{z};z\}{\rm d}z^2 , \qquad
\{\tilde{z}, z\} := \frac{\tilde{z}'''(z)}{\tilde{z}'(z)}-\frac{3}{2}
\left(\frac{\tilde{z}''(z)}{\tilde{z}'(z)}\right)^2 ,\label{omega}
\end{gather}
where $\{\tilde{z}, z\}$ is called the Schwarzian derivative.
Under the projective (M\"obius) transformations
\begin{gather*} z \to \tilde{z} = \frac{az+b}{cz+d} , \qquad ad-bc \neq 0 , \end{gather*}
$\Omega$ transforms covariantly, i.e.,
$\Omega(z){\rm d}z^2 = \tilde{\Omega}(\tilde{z}){\rm d}\tilde{z}^2$
because $\{\tilde{z}, z\}=0$. Consequently, any $m$-dif\/ferential of the form
$Q = Q(z){\rm d}z^m$, where $Q(\Omega(z), \nabla\Omega(z), \nabla^2\Omega(z), \ldots)$
is a homogeneous polynomial of degree $m$ with $\deg \nabla^r\Omega = r+2$,
transforms covariantly and an equation of the form $Q=0$ remains invariant under a projective transformation.
It was shown in~\cite{Dubrovin} that equation~\eqref{gch}
corresponds to the projective-invariant equation $Q=0$ for $m=2$, i.e.,
\begin{gather}
\nabla^2\Omega + J\Omega^2=0 \label{Qeqn}
\end{gather}
for any constant $J$. This yields
the following third order dif\/ferential equation for $\eta(z)$:
\begin{gather}
 \eta''' - 6\eta \eta'' + 9\eta'^2 +(J-12)\big(\eta'-\half\eta^2\big)^2 =0 ,
\label{etaeqn}
\end{gather}
which reduces to \eqref{gch} by setting $y=3\eta$ and $108K=(12-J)$. Thus, the Chazy~III and Chazy~XII equations~\eqref{gch} can be interpreted as a certain dif\/ferential polynomial
{\it invariant} of an af\/f\/ine connection in a~one-dimensional complex manifold with a projective structure.

A projective structure on a one-dimensional complex manifold
$M \subseteq \mathbb{CP}^1$ is def\/ined by an atlas
of local coordinates with transition functions given by M\"obius transformations.
In a local coordinate chart, the M\"obius transformations are generated by the
vector f\/ields $\langle \partial_z, z\partial_z, z^2\partial_z \rangle$ isomorphic
to the Lie algebra $\mathfrak{sl}_2(\mathbb{C})$. A nontrivial
representation of $\mathfrak{sl}_2(\mathbb{C})$ is given by the
vector f\/ields $\langle u_1^2\partial_s, u_1u_2\partial_s, u_2^2\partial_s\rangle$
where $u_1(s)$, $u_2(s)$ is a pair of linearly
independent solutions of the complex Schr\"odinger equation
\begin{gather}
u''+\tf14 V(s)u=0 ,\label{ueqn}
\end{gather}
(where the factor $\tf14$ is inserted for convenience). Then the identif\/ication
of the vector f\/ield $\partial_z$ with $u_1^2\partial_s$ induces a local change
of coordinates $s \to z(s)$ on $M$ via the ratio
\begin{gather}
z(s) = \frac{u_2(s)}{u_1(s)} .\label{ratio}
\end{gather}
If $M$ is not simply connected, the monodromy group $G \subset {\rm GL}_2({\mathbb C})$
resulting from the analytic extensions of the pair
$(u_2, u_1) \to (au_2+bu_1, cu_2+du_1)$ along all
possible closed loops in $M$ acts projectively
on the ratio in~\eqref{ratio} via the M\"obius transformations
\begin{gather}
z \to \gamma(z) := \frac{az+b}{cz+d} , \qquad
\begin{pmatrix} a & b \\ c & d\end{pmatrix} \in G.
\label{gamma}
\end{gather}
The projectivized monodromy group
is the quotient group $\Gamma \cong G/\lambda I_2 \subseteq {\rm PSL}_2({\mathbb C})$,
where $I_2$ is the $2 \times 2$ identity matrix. A dif\/ferent choice for
the basis $(u_1,u_2)$ would lead to a dif\/ferent projective structure on $M$.

Note that the Schwarzian derivative is a dif\/ferential invariant of the projective transformation, i.e., $\{\gamma(z);s\}=\{z;s\}$. For a projective structure induced by the Schr\"odinger equation \eqref{ueqn}, $z(s)$ satisf\/ies the third order Schwarzian equation
\begin{subequations}
\begin{gather}
\{z;s\} = \half V(s) ,
\label{schwarz-a}
\end{gather}
which can be {\it linearized} via \eqref{ueqn} and
\eqref{ratio}. It follows that the inverse function $s(z)$ (if it exists
globally) is a projective invariant function on the manifold $M$ with the
automorphism $s(z) = s(\gamma(z))$, $\gamma \in \Gamma$, and satisf\/ies
\begin{gather}
\{s;z\} + \half V(s) s'(z)^2 =0 .
\label{schwarz-b}
\end{gather}
\end{subequations}
An af\/f\/ine connection on the one-dimensional complex manifold $M$ can be def\/ined
uniquely by its trivialization $\eta(s)=0$ in the projective invariant coordinate $s$.
Then the transformation
$s \to z(s)=u_2(s)/u_1(s)$ in terms of the solutions of the Schr\"odinger equation
leads to the following expression of the af\/f\/ine connection in the $z$-coordinate
\begin{gather*}\eta(z) = s''(z)/s'(z)= - z''(s)/z'(s)^2 = 2 u_1(s)u_1'(s) , \end{gather*}
using the transformation rule \eqref{eta}. Furthermore, from
\eqref{omega}, the dif\/ferential $\Omega(z)$ is
then given in terms of the Schr\"odinger potential $V(s)$ as
\begin{gather*} \Omega(z) = \{s;z\}= -\{z;s\}/z'(s)^2 = -\half u_1^4(s) V(s) . \end{gather*}
Consequently, the solutions of the Chazy III and XII equations~\eqref{gch},
which are equivalent to~\eqref{etaeqn}, can be expressed via
the solutions and the potential of the complex Schr\"odinger equation~\eqref{ueqn}. It is worth noting that if the af\/f\/ine connection is trivialized
in a dif\/ferent coordinate~$x(s)$ with ${\rm d}x = f(s){\rm d}s$, then
from~\eqref{eta}, $\eta(z){\rm d}z = {\rm d}\log(f(s)s'(z))$.
We will utilize this fact and the geometric framework discussed above to construct a solution method in Section~\ref{section3} for~\eqref{gch}, by exploiting its relationship to another nonlinear dif\/ferential system described below.

\subsection{The Halphen system and Schwarz triangle functions}\label{section2.2}
In 1881, Halphen considered a slight variant~\cite[p.~1405, equation~(5)]{Halphen1}
of the following nonlinear dif\/ferential system
\begin{gather}
w_1'=-w_2w_3+w_1(w_2+w_3)+\tau^2, \nonumber\\
w_2'=-w_3w_1+w_2(w_3+w_1)+\tau^2, \label{gdh}\\
w_3'=-w_1w_2+w_3(w_1+w_2)+\tau^2, \nonumber
\\ \tau^2 = \alpha^2(w_1-w_2)(w_2-w_3)+\beta^2(w_2-w_1)(w_1-w_3)+
\gamma^2(w_3-w_1)(w_2-w_3) ,\nonumber \end{gather}
for functions $w_i(z) \neq w_j(z)$, $i\neq j$, where $i,j \in \{1,2,3\}$, and
$\alpha$, $\beta$, $\gamma$ are constants.
A special case of this equation with $\alpha = \beta = \gamma = 0$ originally
appeared in Darboux's work of triply orthogonal surfaces on $\mathbb{R}^3$ in
1878 \cite{Darboux}. Its solution was given by Halphen~\cite{Halphen}
in 1881 in terms of hypergeometric functions. Subsequently, Chazy~\cite{Chazy3}
showed that $y(z):=2(w_1+w_2+w_3)$ satisf\/ies the Chazy~III equation
introduced in Section~\ref{section1}.
More recently, \eqref{gdh} was re-discovered as a certain symmetry
reduction of the self-dual Yang--Mills equations, and was referred to as the
generalized Darboux--Halphen (gDH) system~\cite{Ablowitz-H1,Ablowitz-H2}.

The gDH system can be solved via the Schr\"odinger equation~\eqref{ueqn} with the potential
\begin{gather}
V(s)=\frac{1-\alpha^2}{s^2}+\frac{1-\beta^2}{(s-1)^2}+ \frac{\alpha^2+\beta^2-\gamma^2-1}{s(s-1)} , \label{V}
\end{gather}
and def\/ining a projective structure on $M$ via the ratio~$z(s)$ in \eqref{ratio} as described in Section~\ref{section2.1}.
Note that in this case \eqref{ueqn} is a second order Fuchsian dif\/ferential
equation with three regular singular points, and
$\alpha$, $\beta$, $\gamma$ are the exponent dif\/ferences (for any
pair of linearly independent solutions $u_1$ and $u_2$) prescribed at the
singular points $0$, $1$ and $\infty$, respectively. The generators of the
projectivized monodromy group $\Gamma$ are determined by the exponent
dif\/ferences. If the gDH variables $w_1(z)$, $w_2(z)$, $w_3(z)$
are expressed in terms of the projective invariant inverse
function~$s(z)$ (and its derivatives) as follows:
\begin{gather}
w_1 = \frac{1}{2}\left[\log\left(\frac{s'}{s} \right)\right]', \quad
w_2 = \frac{1}{2}\left[\log\left(\frac{s'}{s-1} \right)\right]', \quad
w_3 = \frac{1}{2}\left[\log\left(\frac{s'}{s(s-1)} \right)\right]',\label{wdef}
\end{gather}
then a straightforward calculation shows that \eqref{gdh} reduces
to the Schwarzian equation \eqref{schwarz-b} for $s(z)$, where the constants
$\alpha$, $\beta$, $\gamma$ in $V(s)$ are the
same as those appearing in $\tau^2$ of~\eqref{gdh}.
Equation~\eqref{schwarz-b} is equivalent to~\eqref{schwarz-a} (after interchanging
the dependent and independent variables) which is then reduced to~\eqref{ueqn}.
Thus, the gDH system can be ef\/fectively {\it linearized}
by the Schr\"odinger equation \eqref{ueqn} with the potential~$V(s)$ given by~\eqref{V}.

The ratio $z(s)$ in \eqref{ratio} of
any two linearly independent solutions $u_1$, $u_2$ of \eqref{ueqn}
with the potential $V(s)$ given by~\eqref{V} def\/ines a conformal mapping
that was studied extensively by H.A.~Schwarz~\cite{Schwarz1} in~1873.
The map~$z(s)$ is, in general, branched at the regular singular
points $s=0, 1, \infty$.
However, if the parameters $\alpha$, $\beta$, $\gamma$ are either
zero or reciprocals of positive integers, and satisfy
$\alpha+\beta+\gamma < 1$, then the mapping~$z(s)$ def\/ines a plane
region D, which is tessellated by an inf\/inite number of
non-overlapping hyperbolic, circular triangles on the complex $z$-plane.
The interior of each triangle is an image of the upper (lower)-half
$s$-plane under the map $z(s)$ and its analytic extensions, and is
bounded by three circular arcs forming interior angles~$\alpha\pi$,~$\beta\pi$, and $\gamma\pi$ at the vertices $z(0)$, $z(1)$,
and $z(\infty)$, respectively. Two adjacent triangles are obtained via Schwarz ref\/lection
principle, and are images of each other under ref\/lection across
the circular arc that forms their common boundary.
In this case, the inverse~$s(z)$ is a~single-valued, meromorphic, automorphic function whose automorphism
group is the projective monodromy group $\Gamma$ associated with~\eqref{ueqn} and~\eqref{V}. That is, $s(\gamma(z)) = s(z)$
for all $\gamma \in \Gamma$ where~$\gamma(z)$ is the M\"obius transformation
def\/ined in~\eqref{gamma}.
When the exponent dif\/ferences satisfy the conditions prescribed
above, $\Gamma$ is a discrete subgroup of ${\rm PSL}(2,{\mathbb R})$, and turns out to be
the group of M\"obius transformations generated by an
even number of ref\/lections across the boundaries of the circular triangles.
This automorphic inverse function
\begin{gather*}s(z) := S(\alpha, \beta, \gamma; z)\end{gather*}
is called the Schwarz triangle function, and the automorphism group~$\Gamma$ is referred to as the triangle group.
It is worth noting that if $\alpha$, $\beta$, $\gamma$ in~\eqref{V} are either zero or reciprocals of positive integers,
but either $\alpha+\beta+\gamma =1$ or $\alpha+\beta+\gamma >1$,
then the map~$z(s)$ in~\eqref{ratio} tiles the $z$-plane
either into inf\/initely many plane triangles, or the extended $z$-plane
(Riemann sphere) into a f\/inite number of spherical triangles, respectively.
The triangle group $\Gamma$ is a simply or doubly periodic
group when $\alpha+\beta+\gamma =1$, and corresponds to one of the
four symmetry groups of the regular solids when $\alpha+\beta+\gamma >1$.
A detailed discussion of the automorphic groups can be found in the monograph~\cite{Ford}.

It follows from \eqref{schwarz-b} or from \eqref{ueqn} with
$V(s)$ as in \eqref{V} that the only possible singularities of $s(z)$
and its derivatives on the domain D are located at the vertices
of each triangle where~$s(z)$ takes the value of~0,~1, or~$\infty$.
The boundary of~D in the $z$-plane is a $\Gamma$-invariant circle~C which
is {\it orthogonal} to all three sides of each triangle and its
ref\/lected images. This orthogonal circle~C is the set of limit points
for the automorphic group $\Gamma$, and corresponds to a dense set
of essential singularities which form a {\it natural barrier} for the
function~$s(z)$. In its domain of existence D, the only possible
singularities of $s(z)$ are poles which correspond to the vertices
where $s(z) = \infty$.

In summary, when the parameters $\alpha$, $\beta$, $\gamma$
in \eqref{V} are either zero
or reciprocals of positive integers and $\alpha+\beta+\gamma <1$,
the general solution of \eqref{schwarz-b}
is obtained as the {\it unique} inverse of the ratio
\begin{gather}
z(s) = \frac{a u_2(s)+b u_1(s)}{cu_2(s) +du_1(s)}, \qquad
a,b,c,d \in {\mathbb C}, \qquad ad-bc =1,
\label{zdef}
\end{gather}
where $u_1$ and $u_2$ are two linearly independent solutions of \eqref{ueqn}.
The solution is single-valued and meromorphic inside a disk in the extended
$z$-plane, and can not be continued analytically across the boundary of
the disk. This boundary is movable as its center and radius are
completely determined by the initial conditions, which depend
on the complex parameters~$a$,~$b$,~$c$,~$d$.
Recently, a number of new nonlinear dif\/ferential equations
whose solutions possess movable natural boundaries have been
found~\cite{SC13, Maier11}.
These can be solved by f\/irst transforming them
into a~Schwarzian equation~\eqref{schwarz-b} and then following the
{\it linearization} scheme described above.

\section{Chazy XII and triangle functions}\label{section3}
In this section we present a solution method for the Chazy XII equation
based on the geometric approach discussed in Section~\ref{section2}. Specif\/ically, we exploit
the fact that the Chazy XII equation in~\eqref{gch} is equivalent to
the coordinate independent form in~\eqref{Qeqn}
for the quadratic dif\/ferential~$\Omega$. We express the af\/f\/ine connection $\eta$ associated with~$\Omega$
via the gDH variables introduced in Section~\ref{section2.2} so that the Chazy~XII
equation can be expressed in terms of the Schr\"odinger potential~$V(s)$
in~\eqref{V} and its derivatives, in a simple algebraic fashion. Then the
solution of the Chazy XII equation is obtained via the
Schwarz triangle function $S(\alpha, \beta, \gamma; z)$ for
specif\/ic values of the triple $(\alpha, \beta, \gamma)$, which we identify.
Recall that in this article we only consider the Chazy~XII equation.
The parametrization of the Chazy~III equation by Schwarz
triangle functions was studied in~\cite{CA09}.

\subsection{The gDH system and the Chazy XII equation}\label{section3.1}
Note f\/irst from \eqref{wdef} and \eqref{eta} that
$2w_i={\rm d}\log(f_i(s)s'(z))/{\rm d}z$ transforms as an af\/f\/ine connection
that is trivialized in the coordinate $x_i(s)$ with ${\rm d}x_i=f_i(s){\rm d}s$ for
\begin{gather*}f_1(s)= \frac{1}{s}, \qquad f_2(s)=\frac{1}{s-1}, \qquad
f_3(s)=\frac{1}{s(s-1)} .\end{gather*}
In particular, under a M\"obius transformation $z \to \gamma(z)$
given by \eqref{gamma} with $ad-bc=1$, $w_i$~transforms as
\begin{gather*}
w_i(\gamma(z)) = (cz+d)^2w_i(z)+ c(cz+d), \qquad \gamma \in \Gamma ,
\end{gather*}
which leaves the gDH system \eqref{gdh} invariant.
Next let us def\/ine the function
\begin{gather}
y(z) := a_1w_1 + a_2w_2 + a_3w_3,\label{ydef}
\end{gather}
in terms of the gDH variables $w_i$, where the coef\/f\/icients $a_i$
are nonnegative constants.
Then from \eqref{wdef}, $y(z)$ can be expressed in terms of~$s(z)$ and its derivatives as follows:
\begin{gather}
y(z) = \frac{p}{2}\, \frac{\phi'(z)}{\phi(z)}, \quad \qquad
\phi(z) = \frac{s'(z)}{s(z)^{1-\alpha_1}(s(z)-1)^{1-\beta_1}},
\label{ypara}
\end{gather}
with $p:= a_1+a_2+a_3$, $\alpha_1:= a_2/p$ and
$\beta_1:= a_1/p$ and $0 \leq \alpha_1, \beta_1 \leq 1$.
If we set $p=6$ above, then it is possible to def\/ine an af\/f\/ine connection
$\eta$ such that in the $z$-coordinate,
\begin{gather*} \eta(z)=\frac{y(z)}{3}=\frac{{\rm d}\log(\phi)}{{\rm d}z}. \end{gather*}
It follows from \eqref{eta} that $\eta$ is trivialized in
the coordinate $x$ such that ${\rm d}x = \phi(z){\rm d}z = f(s){\rm d}s$
where $f(s)=s^{-1+\alpha_1}(s-1)^{-1+\beta_1}$. Moreover,
in the $s$-coordinate,
\begin{gather}
\eta(s)= (\log f)'(s) = -\frac{1-\alpha_1}{s} - \frac{1-\beta_1}{s-1} ,
\qquad \Omega(s) = \eta'(s)-\half \eta^2(s) ,
\label{etas}
\end{gather}
which are rational functions
of $s$. The curvature in the $z$-coordinate is obtained using the
transformation property \eqref{omega} as
\begin{gather*}\Omega(z) = (\Omega(s) - \{z;s\})s'(z)^2 =
(\Omega(s) - \half V(s))s'(z)^2 := -V_2(s)s'(z)^2, \end{gather*}
where the second equality above follows from \eqref{schwarz-a} with
$V(s)$ as in \eqref{V}. Then the covariant derivatives of the
quadratic dif\/ferential $\Omega$ are given by the transformation rule
\begin{gather*}\nabla^m\Omega(z) = -(\nabla^m V_2(s))s'(z)^{m+2}
:= -V_{m+2}(s)s'(z)^{m+2} , \qquad m \geq 1 ,\end{gather*}
where $V_{m+2}(s)$ are rational functions of $s$ def\/ined recursively
for $n \geq 2$ as
\begin{gather}
V_{n+1}(s) = \nabla V_n(s) = V_n'(s) - n\eta(s)V_n(s),
\qquad V_2(s) = \half V(s) - \eta'(s) + \half \eta(s)^2 ,
\label{Vk}
\end{gather}
with $\eta(s)$ def\/ined in \eqref{etas}.

Our next goal is to derive conditions under which $\Omega(z)$ will
satisfy the projective-invariant equation \eqref{Qeqn}, equivalently,
the af\/f\/ine connection $\eta(z)$ will satisfy \eqref{etaeqn}.
Then the function $y(z)=3\eta(z)$ def\/ined in \eqref{ypara} with $p=6$ will
solve the Chazy XII equation in \eqref{gch}. Henceforth, the
value $p=6$ will be used throughout the rest of this article.

It follows from the expressions for $\Omega(z)$ and
its covariant derivatives obtained above, that for $m=2$,
\eqref{Qeqn} implies a simple
algebraic relation between rational functions $V_2(s)$
and $V_4(s)$, namely,
\begin{gather}
V_4(s) = J V_2^2(s), \qquad J \neq 12 ,
\label{Veq}
\end{gather}
that should hold for all $s$. This condition imposes certain restrictions
on the parameters $(\alpha, \beta, \gamma)$ and $(\alpha_1, \beta_1)$
appearing in the functions $V_2(s)$ and $V_4(s)$. It is worth pointing
out here that \eqref{Veq} together with \eqref{ydef}
lead to the central result advertised in Section~\ref{section1}.

In what follows, we
will systematically determine the sets of parameters for which
\eqref{Veq} holds. In particular, we will identify the values of
the triple $(\alpha, \beta, \gamma)$ in the Schwarz triangle
functions $S(\alpha, \beta, \gamma; z)$ which determine the solution
of the Chazy XII equation $y(z)$ via \eqref{ypara}.

\subsection{The Schwarz function parametrization}\label{section3.2}
From equations \eqref{V}, \eqref{etas} and \eqref{Vk},
the rational function $V_2(s)$ can be written as
\begin{gather}
 V_2(s) = \frac{1}{2}\left[\frac{A}{s^2}+\frac{B}{(s-1)^2}+\frac{C}{s(s-1)}\right] ,
\nonumber \\
A= \alpha_1^2-\alpha^2, \qquad B= \beta_1^2-\beta^2, \qquad
C=\gamma_1^2-A-B-\gamma^2, \label{V2}
\end{gather}
where the parameters $(\alpha_1,\beta_1,\gamma_1)$ are def\/ined as follows:
\begin{gather}
\alpha_1=\frac{a_2}{6}, \qquad \beta_1=\frac{a_1}{6} ,
\qquad \gamma_1 = \frac{a_3}{6}, \qquad 0 \leq \alpha_1, \beta_1, \gamma_1 \leq1,
\qquad \alpha_1+\beta_1+\gamma_1=1 .
\label{convex}
\end{gather}
Then $V_4(s)$ is readily computed from the
recurrence relation in \eqref{Vk}. Upon substituting the expressions for $V_2(s)$
and $V_4(s)$ into \eqref{Veq} and rationalizing the resulting expression, one
f\/inds that \eqref{Veq} is satisf\/ied if and only if
\begin{gather*}u_1s^4+u_2s^3(s-1)+u_3s^2(s-1)^2+u_4s(s-1)^3+u_5(s-1)^4=0 \end{gather*}
for all values of $s$. The last identity is equivalent to the vanishing
of the coef\/f\/icients $u_i$, i.e.,{\samepage
\begin{subequations}\label{u}
\begin{gather}
u_1 := JB^2-12B\beta_1^2 =0, \label{u1} \\
u_2 := 2JBC-4[(1-\alpha_1)(1-6\beta_1)B +
\half (1-2\beta_1)(1-3\beta_1)C] =0 , \label{u2} \\
u_3 :=J(2AB+C^2) -4\big[(2-3\beta_1)(1-\beta_1)A+(2-3\alpha_1)(1-\alpha_1)B
\nonumber \\
\hphantom{u_3 :=}{} + \half[(2-3\beta_1)(1-2\alpha_1)+(2-3\alpha_1)(1-2\beta_1)]C\big] =0 ,
\label{u3} \\
u_4 :=2JAC-4[(1-\beta_1)(1-6\alpha_1)A+\half(1-2\alpha_1)(1-3\alpha_1)C]=0 ,
\label{u4} \\
u_5 := JA^2-12A\alpha_1^2=0 . \label{u5}
\end{gather}
\end{subequations}}

\noindent
System \eqref{u} represents a set of coupled, algebraic equations for
the parameters $(\alpha,\beta,\gamma)$, $(\alpha_1, \beta_1)$ and $J$
subject to the conditions
\begin{gather}
\alpha, \beta, \gamma \in \big\{\tf1n , \, n \in \mathbb{N}\big\} \cup \{0\},
\qquad \alpha+\beta+\gamma<1, \qquad 0 \leq \alpha_1, \beta_1 \leq 1,
\label{cond}
\end{gather}
and $J \neq 12$. The case $J=0$ corresponds
to $K=\tf19$ or $k=0$ in \eqref{gch}. The Chazy XII equation for this
case can be linearized via Airy's equation as mentioned in
Section~\ref{section1}. Hence, the $J=0$ case is distinct from
the cases considered here which correspond to Schr\"odinger potential
$V(s)$ in \eqref{V} with three regular singular points. Furthermore,
the following lemma is proven in Appendix~\ref{appendixA}.
\begin{Lemma} \label{lemma1}
If $J=0$, then \eqref{u} has no admissible solution satisfying conditions
\eqref{cond}.
\end{Lemma}
Henceforth, we take $J \neq 0$ through the remainder of the main text of
this paper.
\begin{Lemma} \label{lemma2}
If a triple $(\alpha, \beta, \gamma)$ of non-negative numbers
satisfying $\alpha+\beta+\gamma < 1$
solves the system \eqref{u} with $J \neq 12$,
then $\alpha \neq 0$, $\beta \neq 0$, and $\gamma \neq 0$.
\end{Lemma}
\begin{proof}
If $\alpha=0$, then $A = \alpha_1^2$ from \eqref{V2}. The condition~\eqref{u5}
with $J \neq 12$ implies $\alpha_1=0$. Hence, $A=0$. Then equations
\eqref{u3} and \eqref{u4} imply $B=0$ and $C=0$.
Consequently, $\beta=\beta_1$, $\gamma=\gamma_1$ from~\eqref{V2}.
Therefore, $\beta+\gamma = \beta_1+\gamma_1 =1$ contradicting
$\alpha+\beta+\gamma < 1$.

If $\beta=0$, the proof is similar to above, starting from condition
\eqref{u1} f\/irst, then using equations \eqref{u2} and \eqref{u3}.

If $\gamma=0$ then $D:=A+B+C = \gamma_1^2$. The condition
$\sum\limits_{i=1}^5u_i=0$ in \eqref{u} yields
$JD^2=12\gamma_1^2D$, which implies that $\gamma_1=0$. Hence,
$\alpha_1+\beta_1=1$ and $D=0$. Using the last condition to eliminate~$C$ from \eqref{u2} and~\eqref{u4}, one obtains
\begin{subequations}
\begin{gather}
JAB = (1-2\beta_1)(1-3\beta_1)A-(1-6\beta_1)(1-2\alpha_1-\beta_1)B ,
\label{JAB-a} \\
JAB = -(1-6\alpha_1)(1-2\beta_1-\alpha_1)A + (1-2\alpha_1)(1-3\alpha_1)B ,
\label{JAB-b}
\end{gather}
\end{subequations}
where \eqref{u1} and \eqref{u5} are also used to eliminate terms
involving $A^2$ and $B^2$. It follows from~\eqref{JAB-a} and~\eqref{JAB-b} together
with $\alpha_1+\beta_1=1$ that if either $A=0$,
or $B=0$, then both~$A$ and~$B$ vanish. Hence, $\alpha+\beta =
\alpha_1+\beta_1=1$, which contradicts $\alpha+\beta+\gamma < 1$.

Finally, let $A \neq 0$ and $B\neq 0$. Then
\eqref{u5} and \eqref{u1} imply that
$JA=12\alpha_1^2$, $JB=12 \beta_1^2$.
Substituting these into~\eqref{JAB-a} and~\eqref{JAB-b} to
eliminate~$A$ and~$B$, yields two equations
for~$\alpha_1$ and~$\beta_1$. The resulting equations have no solution
$(\alpha_1,\beta_1)$ that satisf\/ies the condition $\alpha_1+\beta_1=1$.
Thus, $\gamma$ can not be 0.
\end{proof}
One can also make the inequalities for $\alpha_1$, $\beta_1$, $\gamma_1$
in \eqref{convex} strict, i.e., $0 < \alpha_1, \beta_1, \gamma_1 < 1$.
\begin{Lemma} \label{lemma3}
If $\alpha_1$, $\beta_1$, $\gamma_1$ satisfy \eqref{convex} and
solve the system \eqref{u} with $(\alpha,\beta,\gamma)$ as in Lemma~{\rm \ref{lemma2}},
then $\alpha_1, \beta_1, \gamma_1 \notin \{0,1\}$.
\end{Lemma}
\begin{proof}
If $\alpha_1=0$, then it follows from \eqref{u5} that $A=0$.
Then, \eqref{V2} implies that $\alpha = \alpha_1=0$, contradicting Lemma~\ref{lemma2}.
Thus $\alpha_1 \neq 0$. A similar argument starting from \eqref{u1}
shows that $\beta_1 \neq 0$.

If $\alpha_1 =1$, then \eqref{convex} implies that $\beta_1=0$, which is
not possible. Similarly, $\beta_1=1$ or $\gamma_1=1$ is not possible.

Finally, assume that $\gamma_1 = 0$. Since $JD^2 = 12\gamma_1^2D$
where $D:= A+B+C$ (see Lemma~\ref{lemma2}), it follows that $D=0$.
Then, \eqref{V2} implies that $\gamma = \gamma_1=0$,
contradicting Lemma~\ref{lemma2}.
\end{proof}

Lemma \ref{lemma2} implies that
$\alpha, \beta, \gamma \in \{\tf{1}{n}, \, n \in \mathbb{N} \}$
in \eqref{cond}.
Let $s(z)=S(\alpha,\beta,\gamma;z)$ be the triangle function
associated with those parameters, and let $z_0$ be
a vertex of a triangle in the domain of existence D of $s(z)$
such that $s(z_0) := s_0 \in \{0,1,\infty\}$. The map $z(s)$ def\/ined
via the equa\-tions~\eqref{ratio}, \eqref{ueqn} and \eqref{V}, behaves
locally near $s=s_0$ as (see, e.g.,~\cite{Nehari})
\begin{gather*}z(s)=z_0+(s-s_0)^{\mu}\psi_1(s), \qquad s_0 \in \{0,1\} \qquad
\mbox{and} \\
 z(s)=z_0+s^{-\mu}\psi_2\big(s^{-1}\big), \qquad s_0=\infty ,\end{gather*}
where $\psi_i$ is analytic near $s_0$ with $\psi_i(s_0) \neq 0$, and
$\mu \in \{\alpha, \beta, \gamma\}$ is the corresponding
exponent dif\/ference at the singular point $s_0$. Consequently, the inverse
$s(z)$ is single-valued function def\/ined locally as
\begin{gather*}s(z) = s_0 + (z-z_0)^m\phi_1(z), \qquad z_0 \in \{z(0), z(1)\}\qquad
\mbox{and} \\ s(z) = (z-z_0)^{-m}\phi_2(z) , \qquad z_0=z(\infty) ,\end{gather*}
where $\phi_i(z)$ is analytic in the neighborhood of $z=z_0$ with
$\phi_i(z_0) \neq 0$, and $m=\mu^{-1} \in \mathbb{N}$.
Thus, $s-s_0$ has a zero of order $\tf1\alpha$, $\tf1\beta$ at the vertices
$z_0=z(0), z(1)$, respectively, and $s(z)$ has a~pole of order
$\tf1\gamma$ at the vertex $z_0=z(\infty)$ of each triangle in the domain~D.
As mentioned earlier in Section~\ref{section2}, it is suf\/f\/icient to examine the behavior
of the function $s(z)$ and its derivatives near the vertices $z(0)$, $z(1)$,
and $z(\infty)$ of just one triangle inside the domain~D. The Schwarz
ref\/lection principle and the automorphic property then ensure
that $s(z)$ will have the same behavior at the vertices of the
ref\/lected triangles in~D.
A~straightforward calculation using the above behavior of~$s(z)$ in~\eqref{ypara} shows that $y(z)$ is a meromorphic function in~D with simple
poles at the vertices~$z(0)$,~$z(1)$,~$z(\infty)$ of each triangle in D with
the residues
\begin{gather*}\operatorname{Res}_{z(0)}=3(\tf{\alpha_1}{\alpha}-1),
\qquad \operatorname{Res}_{z(1)}=3(\tf{\beta_1}{\beta}-1),
\qquad \operatorname{Res}_{z(\infty)}=3(\tf{\gamma_1}{\gamma}-1). \end{gather*}
Of course, at the boundary of the domain~D, $y(z)$ inherits the same natural
barrier of essential singularities as the function~$s(z)$. Note that it is
possible for $y(z)$ to be analytic at a vertex $z_0$ in the interior of~D
provided that the residue vanishes at $z_0$, i.e., if either $\alpha=\alpha_1$,
$\beta=\beta_1$, or $\gamma=\gamma_1$. However, it is impossible for all three
residues to vanish simultaneously because then
$\alpha+\beta+\gamma = \alpha_1+\beta_1+\gamma_1 =1$, which violates the
condition $\alpha+\beta+\gamma < 1$ in~\eqref{cond}.
The discussion above concerning the singularity structure of~$y(z)$
in the interior of its domain of def\/inition~D
can be summarized in the following proposition.
\begin{Proposition} \label{mero}
The solution $y(z)$ of the Chazy XII equation given by~\eqref{ypara} in terms of the triangle function
$s(z)=S(\alpha, \beta, \gamma; z)$ where $\alpha, \beta, \gamma$
satisfy \eqref{cond}, is meromorphic with only simple poles
inside its domain of definition {\rm D}. These poles can
only occur at the vertices~$z(0)$,~$z(1)$ and~$z(\infty)$ of the
circular triangles tessellating the domain~{\rm D}. Moreover, $y(z)$
must have at least one simple pole at one of the vertices of each triangle.
\end{Proposition}
Hence, there are three distinct cases resulting from Proposition~\ref{mero}, namely, $y(z)$ has a simple pole at
(1)\, only one vertex $z_0 \in \{z(0),z(1),z(\infty)\}$,
(2)\, only two of the three vertices, and (3)\, all three vertices.
The admissible parameters that satisfy \eqref{u}
can be determined by considering each case separately.

{\bf Case 1.} Suppose $y(z)$ is analytic at $z(0)$ and $z(1)$, and
has a simple pole only at $z(\infty)$ on each triangle in D. Then the
vanishing of residues at $z(0)$ and $z(1)$ is equivalent to $\alpha=\alpha_1$
and $\beta=\beta_1$, which implies that $A=B=0$. Note however that $C\neq 0$
since $\operatorname{Res}_{z(\infty)} \neq 0$. In this case~\eqref{u1} and~\eqref{u5} are identically satisf\/ied, while~\eqref{u2} and~\eqref{u4} imply
that $(1-2\alpha_1)(1-3\alpha_1)=(1-2\beta_1)(1-3\beta_1)=0$.
Hence,
\begin{gather*}(\alpha, \beta) \in \left\{\big(\tf12,\tf12\big), \big(\tf12,\tf13\big), \big(\tf13,\tf12\big)
\big(\tf13, \tf13\big)\right\}.\end{gather*}
Since the subcase $\alpha=\beta=\tf12$ is not admissible due to
the condition $\alpha+\beta+\gamma<1$, and the subcases
$(\alpha, \beta)\in\{(\tf12,\tf13), (\tf13,\tf12)\}$
are the same modulo a vertex permutation, there are only two distinct subcases,
namely, $(\alpha, \beta)=(\tf12,\tf13)$ and $(\alpha, \beta)=(\tf13,\tf13)$.

If $(\alpha, \beta)=(\tf12,\tf13)$, the remaining condition \eqref{u3}
yields $JC=\tf13$. Choosing $\gamma=\tf1k$, $k\in \mathbb{N}$,
gives $J=12k^2/(k^2-36)$, which is consistent with the value of $K$
in~\eqref{gch}. Recall (from the def\/inition below~\eqref{etaeqn})
that $K=(12-J)/108$.

If $(\alpha, \beta)=(\tf13,\tf13)$, then \eqref{u3} implies that
$JC=\tf43$. In this case, choosing $\gamma=\tf1k$, $ k\in \mathbb{N}$
gives $J=12k^2/(k^2-9)$ so that $K=1/(9-k^2)$, which is dif\/ferent from
the value used by Chazy in~\eqref{gch}. However, the rescaling
$k \to \tf{k}{2}$ restores the value of $K$ in~\eqref{gch} while modifying
the value of $\gamma$ to $\gamma=\tf{2}{k}$.
If the simple pole of $y(z)$ is chosen to be at $z(0)$ or $z(1)$ instead,
the values of the parameters are simply permuted. Thus, modulo permutations,
there are two possible sets of admissible parameters:
\begin{alignat*}{3}
& (i) \quad && (\alpha, \beta, \gamma) = \big(\tf12, \tf13, \tf1k\big),
\qquad (\alpha_1, \beta_1, \gamma_1) = \big(\tf12, \tf13, \tf16\big), & \\
& (ii)\quad && (\alpha, \beta, \gamma) = \big(\tf13, \tf13, \tf2k\big),
\qquad (\alpha_1, \beta_1, \gamma_1) = \big(\tf13, \tf13, \tf13\big) ,
\end{alignat*}
for the same value of $K$ as in \eqref{gch}, and the rescaled
($k \to 2k$) version of subcase (ii) above.

{\bf Case 2.} Suppose now $y(z)$ is analytic at only one vertex $z(0)$ and
has a simple pole at each of $z(1)$ and $z(\infty)$. Then $\alpha=\alpha_1$,
which means $A=0$, so that \eqref{u5} is automatically satisf\/ied.
Here, $B \neq 0$ and $B+C \neq 0$ since the residues at $z(1)$ and
$z(\infty)$ do not vanish. Then \eqref{u4} leads to three
subcases: (i) $\alpha_1=\tf12$, (ii) $\alpha_1=\tf13$, and
(iii) $C=0$.

In the case where $\alpha_1=\half$, one solves f\/irst for $\beta_1$, then for $B$
and $B+C$ from the remaining equations in~\eqref{u}. The admissible
values of the parameters up to a permutation, are $\beta_1=\tf16$, and
$(\alpha, \beta, \gamma) = (\tf12, \tf1k, \tf2k)$, $k\in \mathbb{N}$
with the value of $K$ as in~\eqref{gch}. However, one also obtains
$(\alpha, \beta, \gamma) = (\tf12, \tf{1}{2k}, \tf1k)$ with $K=1/(9-k^2)$
as in Case 1 above, which can be rescaled via $k \to \tf{k}{2}$ to
Chazy's choice for $K$.

If $\alpha_1=\tf13$, a similar analysis as above yields $\beta_1=\tf16$,
and $(\alpha, \beta, \gamma) = (\tf13, \tf1k, \tf3k)$, $k \in \mathbb{N}$
with $K$ as in \eqref{gch}. There is also a second case with
$(\alpha, \beta, \gamma) = (\tf13, \tf{1}{3k}, \tf1k)$ and $K=4/(36-9k^2)$,
which can be rescaled back to the previous case via $k \to \tf{k}{3}$.

Finally, if $C=0$, \eqref{u3} reduces to $(2-3\alpha_1)(1-\alpha_1)B=0$. Therefore, $\alpha_1=\tf23$ since $\alpha_1 = \alpha \neq 1$ and $B\neq 0$.
The remaining equations in~\eqref{u} yield $\beta_1=\tf16$, $\beta=\gamma=\tf1k$, $k\in \mathbb{N}$ with $K$ as in~\eqref{gch}.

{\bf Case 3.} If $y(z)$ has a simple pole at each of the vertices $z(0)$, $z(1)$
and $z(\infty)$, then $A \neq 0$, $B \neq 0$, and $A+B+C \neq 0$. The equations~\eqref{u1} and~\eqref{u5} imply, respectively, that $JB=12\beta_1^2$ and $JA=12\alpha_1^2$. Substituting these into~\eqref{u2} and~\eqref{u4}, one obtains
\begin{gather*}
(1-6\beta_1)\big[2(1-\alpha_1)B+(1+\beta_1)C\big]=0 , \\
(1-6\alpha_1)\big[2(1-\beta_1)A+(1+\alpha_1)C\big]=0 .
\end{gather*}
The above equations lead to four distinct subcases:
(i)~$\alpha_1=\beta_1=\tf16$, (ii)~$\alpha_1=\tf16$, $\beta_1 \neq \tf16$,
(iii)~$\alpha_1 \neq \tf16$, $\beta_1=\tf16$, (iv) $\alpha_1 \neq \tf16,
\beta_1 \neq \tf16$. From the remaining condition~\eqref{u3} and
the above equations one f\/inds that the f\/irst three subcases yield, modulo
permutations, only one distinct set of admissible parameter values given by
\begin{gather*}(\alpha, \beta, \gamma)= \big(\tf1k, \tf1k, \tf4k\big), \qquad k\in \mathbb{N},
\qquad (\alpha_1, \beta_1, \gamma_1)= \big(\tf16, \tf16, \tf23\big) ,\end{gather*}
with $K$ as in \eqref{gch}. Like the previous two cases, an additional
set $(\alpha,\beta,\gamma)= (\tf{1}{4k}, \tf{1}{4k}, \tf{1}{k})$
with $K=1/(9-4k^2)$ is also found by f\/irst choosing $\gamma=\tf1k$ and
then solving for $\alpha$, $\beta$. This case can be reduced to the
one displayed above by the rescaling $k \to \tf k4$.

\looseness=1 Subcase (iv) leads to $\alpha_1=\beta_1$ and $\alpha=\beta$. Using
condition~\eqref{u3}, one obtains the equilateral triangle
corresponding to $\alpha=\beta=\gamma=\tf2k$, $k\in \mathbb{N}$
and $\alpha_1=\beta_1=\gamma_1=\tf13$ with $K$ as in~\eqref{gch}.
It is also possible to choose $\alpha=\beta=\gamma=\tf1k$ for the
same values of $\alpha_1$ and $\beta_1$ but with $K=1/(9-k^2)$. These
are related by the rescaling $k \to \tf k2$.
\begin{table}[h]
\centering
\begin{tabular}{|c|c|c|c|c|} \hline
Case & $(\alpha,\beta, \gamma)$ & $(\alpha_1, \beta_1)$ & $K$ &
$(\operatorname{Res}_{z(0)}, \operatorname{Res}_{z(1)},\operatorname{Res}_{z(\infty)})$\tsep{2pt}\bsep{2pt}\\ \hline
1(a) & $(\tf12,\tf13,\tf1k)$ & $(\tf12,\tf13)$ & $\tf{4}{36-k^2}$ &
$(0,0,\tf{k-6}{2})$ \tsep{2pt}\bsep{2pt}\\ \hline
1(b) & $(\tf13,\tf13,\tf2k)$ & $(\tf13,\tf13)$ & $\tf{4}{36-k^2}$ &
$(0,0,\tf{k-6}{2})$ \tsep{2pt}\bsep{2pt}\\ \hline
1(b)$^*$ & $(\tf13,\tf13,\tf1k)$ & $(\tf13,\tf13)$ & $\tf{1}{9-k^2}$ &
$(0,0,k-3)$ \tsep{2pt}\bsep{2pt}\\ \hline
2(a) & $(\tf12,\tf1k,\tf2k)$ & $(\tf12,\tf16)$ & $\tf{4}{36-k^2}$ &
$(0,\tf{k-6}{2},\tf{k-6}{2})$ \tsep{2pt}\bsep{2pt}\\ \hline
2(a)$^*$ &$(\tf12,\tf{1}{2k},\tf1k)$ & $(\tf12,\tf16)$ & $\tf{1}{9-k^2}$ &
$(0,k-3,k-3)$ \tsep{2pt}\bsep{2pt}\\ \hline
2(b) & $(\tf13,\tf1k,\tf3k)$ & $(\tf13,\tf16)$ & $\tf{4}{36-k^2}$ &
$(0,\tf{k-6}{2},\tf{k-6}{2})$ \tsep{2pt}\bsep{2pt}\\ \hline
2(b)$^*$ & $(\tf13,\tf{1}{3k},\tf1k)$ & $(\tf13,\tf16)$ & $\tf{4}{9(4-k^2)}$ &
$(0,\tf{3k-6}{2},\tf{3k-6}{2})$ \tsep{2pt}\bsep{2pt}\\ \hline
2(c) & $(\tf23,\tf1k,\tf1k)$ & $(\tf23,\tf16)$ & $\tf{4}{36-k^2}$ &
$(0,\tf{k-6}{2},\tf{k-6}{2})$ \tsep{2pt}\bsep{2pt} \\ \hline
3(a) & $(\tf1k, \tf1k,\tf4k)$ & $(\tf16,\tf16)$ & $\tf{4}{36-k^2}$ &
$(\tf{k-6}{2},\tf{k-6}{2},\tf{k-6}{2})$ \tsep{2pt}\bsep{2pt}\\ \hline
3(a)$^*$ & $(\tf{1}{4k}, \tf{1}{4k},\tf1k)$ & $(\tf16,\tf16)$ & $\tf{1}{9-4k^2}$ &
$(2k-3,2k-3,2k-3)$ \tsep{2pt}\bsep{2pt}\\ \hline
3(b) & $(\tf2k,\tf2k,\tf2k)$ & $(\tf13, \tf13)$ & $\tf{4}{36-k^2}$ &
$(\tf{k-6}{2},\tf{k-6}{2},\tf{k-6}{2})$ \tsep{2pt}\bsep{2pt}\\ \hline
(3b)$^*$ & $(\tf1k,\tf1k,\tf1k)$ & $(\tf13,\tf13)$ & $\tf{1}{9-k^2}$ &
$(k-3,k-3,k-3)$ \tsep{2pt}\bsep{2pt}\\ \hline
\end{tabular}
\caption{Parameters of triangle functions associated with the Chazy XII solution.}\label{t:chazy}
\end{table}
The results of our classif\/ication of the Chazy XII solution in terms of the
Schwarz triangle functions $S(\alpha,\beta,\gamma;z)$ are summarized
in Table~\ref{t:chazy}. The parameters $(\alpha,\beta,\gamma)$ of the
triangle function are listed in the second column, while the parameters
$(\alpha_1,\beta_1)$ in the third column determine the corresponding
solution $y(z)$ given by~\eqref{ypara}. This solution has a simple pole
at one or more of the vertices~$z(0)$,~$z(1)$,~$z(\infty)$ as stated in
Proposition~\ref{mero}. The residue at each simple pole is listed in the
last column. The values of the Chazy parameter $K$ are given in the
fourth column. The parameter values listed in each row are modulo all
possible permutations of the vertices, which
also permutes the triples $(\alpha,\beta,\gamma)$ and
$(\alpha_1,\beta_1,\gamma_1)$ accordingly, but $K$ is the same for all
permutations. Except for the Case~1(a) which was known to
Chazy~\cite{Chazy3} and Cases~1(b) and~3(b) which were found
in~\cite{Ablowitz-H2} using a dif\/ferent method, the remaining cases are
new to the best of our knowledge. Note that
in some cases, the second column of Table~\ref{t:chazy}
contains a parameter other than the reciprocal of a positive integer.
The corresponding Schwarz function is then not single-valued
in its domain of def\/inition but the Chazy XII solution~$y(z)$
given by~\eqref{ypara} still remains single-valued. Such multi-valued
Schwarz functions are related to a single-valued Schwarz function
via a rational transformation.
For example, $S(\tf23,\tf1k, \tf1k;z)$ in Case~2(c) is related to
the single-valued function $S(\tf12,\tf13, \tf1k;z)$ in Case~1(a)
via a degree-2 rational transformation~\cite{Nehari}
\begin{gather*} S\big(\tf12,\tf13, \tf1k; \epsilon z\big) =
\frac{\left[S(\tf23,\tf1k,\tf1k; z)-2\right]^2}
{4\left[1-S(\tf23, \tf1k,\tf1k; z)\right]} , \qquad
\epsilon = \sqrt[3]{-\tf14} . \end{gather*}
In this case
the associated triangle in the $z$-plane with interior angles
$\{\tf{2\pi}{3}, \tf{\pi}{k}, \tf{\pi}{k}\}$ is divided along
the bisector of the angle $\tf{2\pi}{3}$ into two triangles each with
interior angles $\{\tf{\pi}{2}, \tf{\pi}{3}, \tf{\pi}{k}\}$.
We shall discuss such rational transformations among the Schwarz
triangle functions in Section~\ref{section4.2}, which correspond to decomposition
of a curvilinear triangle into two or more similar sub-triangles. Each of the cases
marked with an asterisk corresponds to a single valued triangle function
but the parameter~$K$ is dif\/ferent from that considered by Chazy.
Each such case can be transformed to one in the immediately
preceding row by a rescaling $k \to \tf{k}{m}$ for $m=2,3$ or~$4$.

It is useful to note\footnote{We thank a referee for this observation.} that if
one rescales the Chazy function by
introducing $Y(z) = y(z)/6$ then \eqref{gch} takes the form
\begin{gather*} Y'''- 12YY''+3Y'\,^2 = K'\big(Y'-Y^2\big)^2, \qquad K'=0 \qquad \mbox{or} \qquad
K'=\frac{864}{36-k^2}.\end{gather*}
Then using \eqref{ydef} and \eqref{convex}, $Y(z)$ can be expressed
as a convex linear combination of the gDH variables as follows
\begin{gather*}Y(z) = \beta_1\omega_1+\alpha_1\omega_2+\gamma_1\omega_3 , \qquad
\alpha_1+\beta_1+\gamma_1=1 ,\end{gather*}
where the possible values of $(\alpha_1,\beta_1)$ are listed in Table~\ref{t:chazy}.

\subsection[Chazy XII and $3 \times 3$ systems]{Chazy XII and $\boldsymbol{3 \times 3}$ systems}\label{section3.3}
We now make a couple of observations regarding the equivalence
between the Chazy XII equation and systems of f\/irst order equations.
The f\/irst is the relationship between the Chazy~XII solution~$y(z)$ and the gDH variables $w_i$ given by~\eqref{ydef}.
Note that for each case in Table~\ref{t:chazy},
there is a~gDH system~\eqref{gdh}
def\/ined by the triple $(\alpha, \beta, \gamma)$. Then from \eqref{ydef}
with coef\/f\/icients given by $a_1 = 6\beta_1$, $a_2=6\alpha_1$ and
$a_3=6\gamma_1$, this gDH system can be reduced to the Chazy XII
equation~\eqref{gch} with the corresponding Chazy parameter $K$ listed
in Table~\ref{t:chazy}. It should be evident from Table~\ref{t:chazy} that
there are several gDH systems corresponding to the same Chazy~XII
equation. This fact is illustrated in Table~\ref{t:gdh}, where we list
the combinations of the gDH variables $w_i$ which give the same~$y(z)$
satisfying the Chazy XII equation with parameter $K=\tf{4}{36-k^2}$.
However, in each case the $w_i$ satisfy the gDH system~\eqref{gdh}
with a dif\/ferent function $\tau^2$ parameterized by the triple
$(\alpha,\beta,\gamma)$ listed in the corresponding row of Table~\ref{t:gdh}.
\begin{table}[h]
\centering
\begin{tabular}{|c|c|c|c|c|} \hline
Case & $(\alpha,\beta, \gamma)$ & $(\alpha_1, \beta_1)$
& $y=a_1w_1+a_2w_2+a_3w_3$ &
$\phi(z)=s'(z)/[s^{1-\alpha_1}(s-1)^{1-\beta_1}]$
\tsep{2pt}\bsep{2pt}\\ \hline
1(a) & $(\tf12,\tf13,\tf1k)$ & $(\tf12,\tf13)$ &
$2w_1+3w_2+w_3$ & $\phi=s'/[s^{\frac 12}(s-1)^{\frac 23}]$ \tsep{2pt}\bsep{2pt}\\ \hline
1(b) & $(\tf13,\tf13,\tf2k)$ & $(\tf13,\tf13)$ &
$2w_1+2w_2+2w_3$ & $\phi=s'/[s^{\frac 23}(s-1)^{\frac 23}]$ \tsep{2pt}\bsep{2pt}\\ \hline
2(a) & $(\tf12,\tf1k,\tf2k)$ & $(\tf12,\tf16)$ &
$w_1+3w_2+2w_3$ & $\phi=s'/[s^{\frac 12}(s-1)^{\frac 56}]$ \tsep{2pt}\bsep{2pt}\\ \hline
2(b) & $(\tf13,\tf1k,\tf3k)$ & $(\tf13,\tf16)$ &
$w_1+2w_2+3w_3$ & $\phi=s'/[s^{\frac 23}(s-1)^{\frac 56}]$ \tsep{2pt}\bsep{2pt}\\ \hline
2(c) & $(\tf23,\tf1k,\tf1k)$ & $(\tf23,\tf16)$ &
$w_1+4w_2+w_3$ & $\phi=s'/[s^{\frac 13}(s-1)^{\frac 56}]$ \tsep{2pt}\bsep{2pt}\\ \hline
3(a) & $(\tf1k, \tf1k,\tf4k)$ & $(\tf16,\tf16)$ &
$w_1+w_2+4w_3$ & $\phi=s'/[s^{\frac 56}(s-1)^{\frac 56}]$ \tsep{2pt}\bsep{2pt}\\ \hline
3(b) & $(\tf2k,\tf2k,\tf2k)$ & $(\tf13, \tf13)$ &
$2w_1+2w_2+2w_3$ & $\phi=s'/[s^{\frac 23}(s-1)^{\frac 23}]$\tsep{2pt}\bsep{2pt}\\ \hline
\end{tabular}

\caption{Distinct gDH systems associated with Chazy XII equation with $K=\tf{4}{36-k^2}$.}\label{t:gdh}
\end{table}

The last column of Table~\ref{t:gdh} lists the function $\phi(z)$ whose
logarithmic derivative in \eqref{ypara} gives the solution~$y(z)$ for the Chazy XII equation. Since $\phi(z)$ are expressed in terms
triangle functions $S(\alpha,\beta,\gamma;z)$ and their derivatives,
it is possible for the same solution $y(z)$ to be parametrized by
two dif\/ferent triangle functions $s_1(z)$ and $s_2(z)$. For instance,
Cases~1(a) and~1(b) correspond to $s_1(z)=S(\half,\tf13,\tf1k; z)$
and $s_2(z)=S(\tf13,\tf13,\tf2k; z)$. Then it follows from~\eqref{ypara}
that the correspon\-ding~$\phi(z)$ must be proportional. This leads to a dif\/ferential relation between the two triangle functions, namely,
\begin{gather*} \frac{s_1'(z)}{s_1^{1/2}(s_1-1)^{2/3}} = C \frac{s_2'(z)}{s_2^{2/3}(s_2-1)^{2/3}} , \end{gather*}
for some constant $C$, and can be solved to yield
$s_1(z) = (2s_2(\epsilon z)-1)^2$, $\epsilon=C4^{-\frac 13}$.
In fact, all the triangle functions listed in Table~\ref{t:gdh} are related via
rational transformations, which will be deduced in the next section
by employing certain well-known transformations between solutions
of hypergeometric equations.

Lastly, we introduce a Ramanujan-type system that is equivalent
to the Chazy XII equation. Recall
from Section~\ref{section1} that the Ramanujan system \eqref{dPQR} is equivalent
to the Chazy III equation given by \eqref{gch} with $K=0$ if one
identif\/ies $y(z)=\pi i P(q)$. It will be shown that both \eqref{dPQR} and the
dif\/ferential system \eqref{dhPQR} introduced below have a simple geometric
interpretation. Using the notations in Section~\ref{section2.1}, let us def\/ine
\begin{gather}
\hat{P}(z) := \frac{3}{\pi i}\eta(z), \qquad
\hat{Q}(z) := \frac{18}{\pi^2}\Omega(z),
\qquad \hat{R}(z) := \frac{27i}{\pi^3}\nabla \Omega(z) .
\label{hPQR}
\end{gather}
Then, if $\eta(z)$ satisf\/ies \eqref{etaeqn},
the triple $(\hat{P}, \hat{Q}, \hat{R})$ satisf\/ies the dif\/ferential system
\begin{gather}
\frac{1}{2 \pi i} \hat{P}'(z) =\frac{\hat{P}^2-\hat{Q}}{12}, \qquad
\frac{1}{2 \pi i}\hat{Q}'(z)=\frac{\hat{P}\hat{Q}-\hat{R}}{3}, \qquad
\frac{1}{2 \pi i} \hat{R}'(z) =\frac{\hat{P}\hat{R}}{2}-\frac{J}{24}\hat{Q}^2.
\label{dhPQR}
\end{gather}
The f\/irst two equations in \eqref{dhPQR} are simply the def\/initions
of the quadratic dif\/ferential $\Omega=\eta'-\eta^2/2$ and its covariant
derivative $\nabla \Omega = \Omega' - 2\eta \Omega$ associated with the
af\/f\/ine connection~$\eta(z)$. The third one is
equation \eqref{Qeqn} which is equivalent
to the Chazy XII equation with $J=12-108K$, $K$ being
the Chazy parameter. When $J=12$, \eqref{dhPQR} reduces to the Ramanujan
system \eqref{dPQR} which is equivalent to the Chazy~III equation.
More explicitly, if $y(z)$ is a solution of~\eqref{gch}
then
\begin{gather*} \hat{P}= \frac{y}{\pi i}, \qquad \hat{Q} = - \frac{6y'-y^2}{(\pi i)^2} ,
\qquad \hat{R} = \frac{9y''-9yy'+y^3}{(\pi i)^3} \end{gather*}
satisfy \eqref{dhPQR} which subsumes the Ramanujan system \eqref{dPQR}.
The Ramanujan triple $(P, Q, R)$ has a {\it modular} interpretation since
it is related to the Eisenstein series associated with the modular
group ${\rm SL}_2(\mathbb{Z})$. The functions $\hat{P}$, $\hat{Q}$, $\hat{R}$ are also automorphic
forms for the triangle group $\Gamma$ and are parameterized by the triangle
functions $S(\alpha, \beta, \gamma; z)$ via~\eqref{ypara}. In fact
they transform as follows:
\begin{gather*}
\hat{P}(\gamma(z)) = (cz+d)^2 \hat{P}(z) + \frac{6c(cz+d)}{\pi i} ,
\quad \qquad \hat{Q}(\gamma(z)) = (cz+d)^4\hat{Q}(z) , \\
\hat{R}(\gamma(z)) = (cz+d)^6\hat{R}(z) , \qquad
\text{where} \qquad \gamma(z) = \frac{az+b}{cz+d}, \qquad
\begin{pmatrix} a & b \\ c & d \end{pmatrix} \in \Gamma .
\end{gather*}
However, we are not aware of any deep automorphic interpretation
for $(\hat{P}, \hat{Q}, \hat{R})$
similar to that of the Ramanujan triple.
Ramanujan~\cite{Ramanujan-notebook} also gave
a parametrization of his triple using the hypergeometric
function $_2F_1(\half,\half,1; s)$ that is related to the
complete elliptic integral of the f\/irst kind. In the following section,
we discuss the parametrizations of the Chazy XII solution $y(z)$ as
well as $(\hat{P}, \hat{Q}, \hat{R})$ in terms of
hypergeometric functions.

\section{Parameterization of Chazy XII solutions}\label{section4}
Explicit solutions of the Chazy equation were
presented in terms of the triangle functions listed in Table~\ref{t:chazy}
of Section~\ref{section3}. Recall that the triangle functions satisfy the
nonlinear third order equation given by \eqref{schwarz-b}
with the potential $V(s)$ as in \eqref{V}. Equation \eqref{schwarz-b}
is linearized via solutions of the Fuchsian equation \eqref{ueqn} associated
with~$V(s)$ using \eqref{zdef}.
Consequently, it is more convenient to express the Chazy solution $y(s(z))$
{\it implicitly}, that is, in terms of the variable~$s$ and a solution~$u(s)$
of the linear equation~\eqref{ueqn}. Thus it is possible to treat the
solutions of the nonlinear Chazy~XII equation in terms of the classical
theory of linear Fuchsian dif\/ferential equations with three
regular singular points,
equivalently, via the hypergeometric equation. This is the main purpose
of the present section.

\subsection{Hypergeometric parametrization}\label{section4.1}
In the following the domain D of the triangle functions $s(z)$ will be taken
as the interior of the orthogonal circle C discussed in Section~\ref{section2.2},
and the hypergeometric form of
the Fuchsian dif\/ferential equation \eqref{ueqn} will be considered, in order
to make contact with standard literature.
If $u(s)$ is a solution of \eqref{ueqn}, then the function
\begin{gather}
 \chi(s) = s^{(\alpha-1)/2}(s-1)^{(\beta-1)/2}u(s)
\label{chi}
\end{gather}
satisf\/ies the hypergeometric equation
\begin{subequations}
\begin{gather}
\chi''+\left(\frac{1-\alpha}{s}+\frac{1-\beta}{s-1}\right)\chi'+
\frac{(\alpha+\beta-1)^2-\gamma^2}{4s(s-1)}\chi=0 ,\label{hyper-1}
\end{gather}
which can be written in more standard form as
\begin{gather}
s(s-1)\chi''+[(a+b+1)s-c]\chi'+ab\chi=0,\label{hyper-2}
\end{gather}
\end{subequations}
where $a=\half(1-\alpha-\beta-\gamma)$, $b=\half(1-\alpha-\beta+\gamma)$,
and $c=1-\alpha$. The transformation \eqref{chi} sets the local exponents
to $(0, \alpha)$ at $s=0$, $(0, \beta)$ at $s=1$ and
$(\half(1-\alpha-\beta-\gamma), \half(1-\alpha-\beta+\gamma))$ at $s=\infty$.
Note, however, that the exponent {\it differences}
as well as the ratio $z(s)$ of any two linearly independent solutions
of~\eqref{hyper-2} coincide with those for~\eqref{zdef}. Consequently,
one can employ the ratio of two independent solutions of~\eqref{hyper-1} instead of~\eqref{ueqn} to construct the
conformal mapping and triangle function described in Section~\ref{section2.2}.

We next outline how to construct the triangle functions $s(z)$
listed in Table~\ref{t:chazy} together with their orthogonal circles using pairs
of linearly independent hypergeometric solutions (see, e.g.,~\cite{Nehari}).
Notice from \eqref{zdef} that $z(s)$ is def\/ined up to an arbitrary M\"obius
transformation with three complex parameters of the $z$-plane
depending on the choice of linearly independent solutions
$\chi_1$ and $\chi_2$ of~\eqref{hyper-2}.
One can then choose two of the three parameters in the M\"obius transformation
in such a manner that the conformal map~\eqref{zdef} results in a triangle which has the vertex~$z(0)$ placed
at the origin of the $z$-plane, and the two circular arcs meeting there
can be transformed to linear segments subtending angle $\pi \alpha$
at this vertex $z(0)$. The remaining freedom in the M\"obius transformation
can be used to rotate the line segment connecting the vertices
$z(0)$ and $z(1)$ onto the real axis. The remaining side
joining $z(1)$ and $z(\infty)$ is formed by the arc of a circle such that
the origin~O is in the exterior of this circle when the sum of the interior
angles of the triangle is less than $\pi$, i.e., when $\alpha+\beta+\gamma < 1$.
Hence, it is possible to draw a line from the origin tangent to this circle
at some point~F (see Fig.~\ref{fig:R}, Appendix~\ref{appendixB}). Then there exists
a unique circle C with center
at the origin~O and passing through the point~F, thus having radius~OF,
that is orthogonal to the straight edges of the triangle
thus constructed. Consequently, C is orthogonal to each side of the
triangle.

A pair of hypergeometric solutions whose ratio
maps the upper half $s$-plane onto a triangle constructed above is
given by $\chi_1 ={}_2F_1(a,b;c;s)$ and $\chi_2 = s^{1-c}\,{}_2F_1(a-c+1,b-c+1;2-c;s)$.
Here the notation of \eqref{hyper-2} has been used and ${}_2F_1(a,b;c;s)$ is
the standard hypergeometric series solution of \eqref{hyper-2}, analytic
in the neighborhood of $s=0$ with $_2F_1(a,b;c;s=0)=1$. Note that
$\chi_2$ vanishes at the branch point $s=0$ if $\alpha=1-c >0$.
The explicit form of the map is then
\begin{gather}
z(s) = \frac{\chi_2}{\chi_1}=\frac{s^{1-c}\,{}_2F_1(a-c+1,b-c+1;2-c;s)}{_2F_1(a,b;c;s)} .
\label{map}
\end{gather}
Since the triples $(\alpha, \beta, \gamma)$
listed in Table~\ref{t:chazy} satisfy \eqref{cond}, $\alpha >0$.
Hence, it is clear that $z(0)=0$.
Furthermore, the parameters $(a,b,c)$ are real
and positive. Consequently, the $_2F_1$ functions are also real and positive
for $\operatorname{Re}(s)>0$. Moreover, if we take the real, positive branch
of $s^{1-c}$ when $\operatorname{Re}(s)>0$, then it follows from~\eqref{map} that~$z(s)$ is also real and positive as $s$ varies from~$0$ to~$1$ along
the real $s$-axis. This conf\/irms that one side of the triangle lies
on the positive $z$-axis joining the vertices $z(0)=0$ and $z(1)$.
The other side of the triangle originating from
the vertex $z(0)=0$ is the conformal image of the negative real $s$-axis
on which the $_2F_1$ functions are also real but the factor
$s^{1-c}=(-|s|)^{1-c} =|s|e^{i\pi\alpha}$. This shows that the negative
real $s$-axis is mapped to a~linear segment joining $z(0)$ and $z(\infty)$,
and making an angle~$\pi \alpha$ with the positive $z$-axis.
One can also compute the vertices~$z(1)$ and~$z(\infty)$ by considering
the analytic continuations of the~$_2F_1$ functions into the neighborhoods
of $s=1$ and $s=\infty$. These are given by (see, e.g.,~\cite{Bateman, Nehari})
\begin{gather*}z(1)=\frac{\Gamma(2-c)\Gamma(c-a)\Gamma(c-b)}{\Gamma(c)\Gamma(1-a)\Gamma(1-b)}, \qquad
z(\infty) = e^{\pi i (1-c)}
\frac{\Gamma(b)\Gamma(c-a)\Gamma(2-c)}{\Gamma(c)\Gamma(b-c+1)\Gamma(1-a)} ,\end{gather*}
where $\Gamma(\cdot)$ is the Gamma function. From the expressions for $z(1)$
and $z(\infty)$, it is possible to determine the radius~$R$ of the
orthogonal circle~C which forms the natural barrier beyond which~$s(z)$
can not be analytically continued. In terms of the triple $(\alpha, \beta, \gamma)$
parameterizing the triangle function $S(\alpha,\beta,\gamma;z)$ the square of the
radius of the barrier is~\cite{Cosgrove}
\begin{gather*} R^2 = \left(\frac{\Gamma(1+\alpha)}{\Gamma(1-\alpha)}\right)^2
\prod_{\epsilon_i=\pm 1}\frac{\Gamma(\half(1-\alpha+\epsilon_1\beta+\epsilon_2\gamma))}
{\Gamma(\half(1+\alpha+\epsilon_1\beta+\epsilon_2\gamma))} . \end{gather*}
The above expression appears in~\cite{Chazy3, Cosgrove} but without a derivation,
which, although elementary, is not immediately obvious. For that reason, we
have included a brief derivation in Appendix~\ref{appendixB}.

The map $z(s)$ in \eqref{map} is a Puiseux series in $s$ of the
form $z(s) = s^{\alpha}\psi(s)$ where $\psi(s)$ is analytic near $s=0$
with $\psi(0) \neq 0$. In fact, the power series for~$\psi(s)$ can
be readily derived from the series expansion of the~$_2F_1$ functions
in \eqref{map}. Since $\alpha = \tf1n$, $n \geq 2$, $n \in \mathbb{N}$,
the series for $z(s)$ can be inverted to obtain the power series
of the inverse $s(z)$ in the form
\begin{gather*} s(z)= z^n\big(1+b_1z^n+b_2z^{2n}+\cdots\big) ,\end{gather*}
where the coef\/f\/icients $b_j$ can be obtained recursively from the
coef\/f\/icients in the series expansion of $\psi(s)$. The series for~$s(z)$
converges in a neighborhood of $z(0)=0$ and def\/ines a single-valued, holomorphic
function in this neighborhood as discussed earlier in Section~\ref{section3.2}.
By analytic continuation of the
hypergeometric functions onto the neighborhoods of $s=1$ and $s=\infty$,
it is possible to obtain similar series expansions for $s(z)$ near
$z(1)$ and $z(\infty)$ as well. Note that $s(z)$ has a pole of order
$\tf1\gamma$ at $z(\infty)$.

Let $(\chi_1, \chi_2)$ be the pair of hypergeometric solutions
whose ratio def\/ines $z(s)$ as in \eqref{map}, then
$s'(z) = 1/z'(s)=\chi_1^2/W(\chi_1, \chi_2)$
where the Wronskian of the pair of solutions
$W(\chi_1, \chi_2) = As^{\alpha-1}(s-1)^{\beta-1}$
from Abel's formula. The nonzero constant $A$ is found by explicitly
calcu\-la\-ting the Wronskian of $(\chi_1, \chi_2)$ in
\eqref{map} and letting $s \to 0$. Thus one obtains $A = (-1)^{1-\beta}\alpha$.
Substituting the expression for $s'(z)$ in~\eqref{ypara} yields
a parametrization for $y(z)$ in terms of~$\chi_1(s)$ and~$\chi_1'(s)$, namely
\begin{gather}
y(s(z)) = \frac{3}{\alpha} s^{1-\alpha}(1-s)^{1-\beta}\left(2\chi_1\chi_1'+\left[\frac{\alpha_1-\alpha}{s}+\frac{\beta_1-\beta}{s-1}\right]
\chi_1^2\right) .\label{yimp}
\end{gather}
To be clear $y(z)$ is expressed parametrically by $z(s)$ in \eqref{map} and $y(s)$ given by \eqref{yimp}.

Note that when $\alpha=\alpha_1$, $\beta=\beta_1$, $y(s(z))$ is analytic
at the vertices $z=z(0), z(1)$ since $\chi_1(s)$ is analytic
at $s=0, 1$. On the other hand, if $\alpha=\alpha_1$ but $\beta \neq \beta_1$
then $y(s(z))$ is analytic at~$z(0)$ but has a simple pole at $z=z(1)$. These
observations are consistent with Cases~1 and~2 in Section~\ref{section3.2}.

The Ramanujan-type triple $(\hat{P}, \hat{Q}, \hat{R})$ of automorphic
functions introduced in Section~\ref{section3.3} can also be formulated in terms of
the hypergeometric function $\chi_1(s) ={}_2F_1(a,b;c;s)$.
The function $\hat{P}=y/(\pi i)$ is obtained directly from~\eqref{yimp} above.
In order to obtain expressions for $\hat{Q}$, $\hat{R}$, one f\/irst recalls
from Section~\ref{section3.1} that the quadratic dif\/ferential $\Omega$ and its
covariant derivatives are given in terms of $s(z)$ and $s'(z)$ by
\begin{gather*}\Omega(s(z)) = -V_2(s(z))s'(z)^2 , \qquad \nabla\Omega(s(z)) = -V_3(s(z))s'(z)^3 ,\end{gather*}
where the rational functions $V_2(s)$ and $V_3(s)$ are obtained from~\eqref{Vk}.
Then from \eqref{hPQR} together with the relation $s'(z) =\chi_1^2/W(s)$ it follows that
\begin{gather}
\hat{Q} = -2\left(\frac{3}{\pi W(s)}\right)^2V_2(s)\chi_1^4(s) , \qquad
\hat{R} = \left(\frac{3i}{\pi W(s)}\right)^3V_3(s)\chi_1^6(s),\label{QRimp}
\end{gather}
where $W(s)=\alpha s^{\alpha-1}(1-s)^{\beta-1}$.
The explicit form of the function $V_2(s)$
is given in \eqref{V2}, and from \eqref{Vk} one obtains
$V_3(s)=V_2'(s)+\big(\tf{1-\alpha_1}{s}+\tf{1-\beta_1}{s-1}\big)V_2(s)$. Alternatively,
one could use \eqref{yimp} and the f\/irst two equations from~\eqref{dhPQR} to derive~\eqref{QRimp}.

Equations \eqref{yimp} and \eqref{QRimp} provide parametrizations
of the triple $(\hat{P}, \hat{Q}, \hat{R})$, equivalently the dif\/ferential
geometric quantities $(\eta, \Omega, \nabla \Omega)$ in terms of~$s$
and the hypergeometric function~$\chi_1(s)$. Conversely,
it is also possible to derive expressions
for~$S(\alpha,\beta,\gamma;z)$ and $\chi_1$
in terms of the functions $\hat{P}$, $\hat{Q}$ and~$\hat{R}$.
For instance, Case~1(a) in Table~\ref{t:chazy} gives rise to the following relations
\begin{gather*} \left(\frac{J\pi^2}{12}\hat{Q}\right)^{\frac14} =
(1-s)^{\frac {1}{12}}\,{}_2F_1\big(\tf{1}{12}-\tf{1}{2k},\tf{1}{12}+\tf{1}{2k};\half; s\big),
\\
s(z)= S\big(\half,\tf13;\tf1k; z\big) = \frac{12 \hat{R}^2}{12 \hat{R}^2-J\hat{Q}^3}. \end{gather*}
Note that if $k \to \infty$ (Chazy~III case with $J=12$),
one recovers the well known representation of the modular function
$S(\half,\tf13;0; z)$ from above in terms
of the Ramanujan functions $Q$ and $R$, which are the Eisenstein series
of weight 4 and 6 respectively, for the modular group. Moreover, the
f\/irst relation above leads to (a slight variant of) a remarkable
identity discovered by Ramanujan~\cite{Ramanujan-collect}
\begin{gather*}\sqrt{\pi}Q^{\frac 14} = {}_2F_1\big(\tf{1}{12},\tf{5}{12};\half; \tf{R^2}{Q^3}\big) .\end{gather*}
This identity is obtained by f\/irst letting $k \to \infty$ and then applying
the Pfaf\/f transformation $(1-s)^{1/12}\,{}_2F_1(\tf{1}{12},\tf{1}{12};\half; s)
= {}_2F_1(\tf{1}{12},\tf{5}{12};\half; \tf{s}{s-1})$.

\subsection{Pull-back maps of Schwarz functions}\label{section4.2}
Recall that Table~\ref{t:gdh} of Section~\ref{section3.3} lists the parametrizations
of the Chazy XII solution $y(z)$ with Chazy parameter $\tf{4}{36-k^2}$
in terms of dif\/ferent Schwarz triangle functions.
It follows from \eqref{ypara} that any two of these Schwarz functions
are related via the dif\/ferential relation
$\phi_1(z) = C \phi_2(z)$ for some constant $C$, where the $\phi_i(z)$
are listed in the last column of Table~\ref{t:gdh}. In fact, the
example given below Table~\ref{t:gdh} provides such a mapping between the triangle
functions corresponding to Cases 1(a) and 1(b) of Table~\ref{t:gdh}. In this
subsection we systematically outline the transformations among the various
triangle functions linking all the cases presented in Table~\ref{t:gdh}.

The mappings of the Schwarz functions stem from the well-known algebraic
transformations of the hypergeometric functions induced by
the pull-back transformation of the corresponding hypergeometric dif\/ferential
equation \eqref{hyper-2}~\cite{Goursat}. These transformations are of the form
${}_2F_1(a,b;c ; s) = \xi(s)\,{}_2F_1(a',b'; c'; \theta(s))$, where
$\theta(s)$ is a rational function, and $\xi(s)$ is an algebraic
function (see, e.g.,~\cite{Bateman}, also~\cite{Vid09} for more recent
work). To make this paper self-contained,
we include here the details of the derivations of the relations between the
triangle functions listed in Table~\ref{t:gdh}. It suf\/f\/ices to derive the
transformations relating Case 1(a) to all other cases in Table~\ref{t:gdh}.
Let $\tilde{s} = \theta(s)$ denote the rational map where
$\tilde{s}=S(\half, \tf13, \tf1k;z)$ corresponds to Case 1(a)
and $s=S(\alpha, \beta, \gamma ;z)$ corresponds to any of the
other cases. Let $\tilde{s}=f_1(s_1)=f_2(s_2)$ denote
two such rational maps for Schwarz functions $s_1(z)$ and $s_2(z)$
which parametrize the same Chazy XII solution $y(z)$. Then the
transformation between $s_1$ and $s_2$ can be expressed as
$s_2 = (f_2^{-1} \circ f_1)(s_1)$, which may generally
not be single valued.

Below we consider each case from Table~\ref{t:gdh} separately,
treating Case (i) below as an illustrative example of the procedure.

{\bf Case (i).} The mapping between the Schwarz functions
$S(\half, \tf13, \tf1k; z)$ and $S(\tf13, \tf13, \tf2k; z)$
corresponding to Cases 1(a) and 1(b) of Table~\ref{t:gdh}
follows from the well-known quadratic transformation~\cite{Goursat}
\begin{gather*}{}_2F_1(a,b;c; s) = {}_2F_1\big(\tf{a}{2},\tf{b}{2}; c; 4s(1-s)\big),
\qquad c = \frac{a+b+1}{2}, \end{gather*}
of hypergeometric functions. Note that the parameters $(a,b,c)$ above
are constrained by a linear relation.
Choosing $a=\tf16 - \tf1k$, $b =\tf16 + \tf1k$
whence $c=\tf23$, yields $(\alpha, \beta, \gamma)=(1-c, c-a-b, b-a)
= (\tf13, \tf13, \tf2k)$ and $(\alpha, \beta, \gamma) =
(1-c, c-\half(a+b), \half(b-a)) = (\tf13, \tf12, \tf1k)$. These values
correspond to the exponent dif\/ferences at the singular points
$s=0,1,\infty$ (up to their permutations).
One applies the quadratic transformation to the
${}_2F_1$ functions in the right hand side of the map~\eqref{map}
to construct~$z(s)$ f\/irst, and then invert the obtained relation
to recover the map between the two triangle functions. Thus, in this case,
one obtains from \eqref{map}
\begin{gather*}z(s) = s^{\frac 13} \frac{_2F_1(\tf12-\tf1k,\tf12+\tf1k; \tf43; s)}
{_2F_1(\tf16 - \tf1k,\tf16 + \tf1k; \tf23; s)} = (s(1-s))^{\frac 13}
\frac{_2F_1(\tf56 - \tf1k,\tf56 + \tf1k; \tf43; s)}
{_2F_1(\tf16 - \tf1k,\tf16 + \tf1k; \tf23; s)} , \end{gather*}
where the Euler transformation
${}_2F_1(a,b;c; s)=(1-s)^{c-a-b}\,{}_2F_1(c-a,c-b;c; s)$ and the symmetry
${}_2F_1(a,b;c; s)={}_2F_1(b,a;c; s)$ have been used to obtain the
last equality. Note that the parameters in the ${}_2F_1$ functions
of this last expression satisfy the relation $c=(a+b+1)/2$.
Next, the quadratic transformation is applied
to both the ${}_2F_1$ functions on the numerator and denominator to yield
\begin{gather*}\tilde{z}(\tilde{s}):=\sqrt[3]{4}z(s) = \tilde{s}^{\frac 13}
\frac{_2F_1(\tf{5}{12} - \tf{1}{2k},\tf{5}{12} + \tf{1}{2k}; \tf43; \tilde{s})}
{_2F_1(\tf{1}{12} - \tf{1}{2k},\tf{1}{12} + \tf{1}{2k}; \tf23; \tilde{s})}
\qquad \tilde{s}:=4s(1-s) . \end{gather*}
Inverting the above relation leads to $\tilde{s} = 4s(1-s)$ where
$\tilde{s}(z) = S(\tf13, \tf12, \tf1k; \sqrt[3]{4}z)$ and
$s(z)= S(\tf13, \tf13, \tf2k; z)$. By making the substitution
$\tilde{s} \to 1-\tilde{s}$, which amounts to switching the singular
points $\tilde{s}=0$ and $\tilde{s}=1$, one obtains
\begin{gather*}\tilde{s}=\theta(s)=1-4s(1-s)=(2s-1)^2, \qquad
\tilde{s}(z)= S\big(\tf12, \tf13, \tf1k; \sqrt[3]{4}z\big),\end{gather*}
as the rational map $\theta(s)$ of degree 2 between the two triangle
functions listed in Cases 1(a) and 1(b) of Table~\ref{t:gdh}.
The associated pull-back transformation is denoted in terms of
the local exponents as
$(\tf12, \tf13, \tf1k) \mapto2 (\tf13, \tf13, \tf2k)$~\cite{Vid09}.

{\bf Case (ii).} Here, the transformation between the Schwarz function
for each of the three subcases in Case 2 of Table~\ref{t:gdh}
and the Schwarz function in Case 1(a) will
be discussed. The mapping between Cases 2(a) and 1(a) is obtained
from the cubic transformation of hypergeometric equation~\cite{Bateman, Vid09}
\begin{gather*} _2F_1\big(a,\tf{1-a}{3};\tf{4a+5}{6};s\big) =
(1-4s)^{-a}\,{}_2F_1\big(\tf{a}{3},\tf{a+1}{3};\tf{4a+5}{6}; \tilde{s}\big), \qquad
\tilde{s}=\frac{27s}{(4s-1)^3}, \end{gather*}
where the ${}_2F_1$ parameters $(a,b,c)$ depend only on $a$. Choosing
$a=\tf14-\tf{3}{2k}$, the above transformation gives the degree 3 map
$(\tf1k, \tf12, \tf13) \mapto3 (\tf1k, \tf12, \tf2k)$.
In this case, \eqref{map} leads to
\begin{gather*} z(s) = s^{\frac{1}{k}} \frac{_2F_1(\tf14+\tf{3}{2k},\tf14-\tf{1}{2k};1+\tf1k;s)}
{_2F_1(\tf14-\tf{3}{2k},\tf{1}{4}+\tf{1}{2k};1-\tf1k;s)} , \end{gather*}
which, after applying the cubic transformation of the hypergeometric
functions, leads to
\begin{gather*}\tilde{z}(\tilde{s}) := (-27)^{\frac 1k}z(s)=\tilde{s}^{\frac 1k}
\frac{_2F_1(\tf{1}{12}+\tf{1}{2k},\tf{5}{12}+\tf{1}{2k};1+\tf1k;\tilde{s})}
{_2F_1(\tf{1}{12}-\tf{1}{2k},\tf{5}{12}-\tf{1}{2k};1-\tf1k;\tilde{s})}.\end{gather*}
Inverting the above relation, one obtains the transformation
between the triangle functions
$\tilde{s}(z)=S(\tf1k,\tf12,\tf13;\tilde{z})=27s/(4s-1)^3$, $s(z)=S(\tf1k,\tf12,\tf2k;z)$. In order to obtain the degree 3 map
relating Case~2(a) to~1(a), one needs
to switch the singular points $s=0$ and $s=1$ so that $s \to 1-s$
in above, and also permute the singular points for $\tilde{s}$
as $(0,1,\infty) \to (\infty, 0, 1)$ by letting $\tilde{s} \to 1-\tilde{s}^{-1}$.
Thus, one obtains the degree 3 map
$(\tf12, \tf13, \tf1k) \mapto3 (\tf12, \tf1k, \tf2k)$ given by
\begin{gather*}\tilde{s}=S\big(\tf12, \tf13, \tf1k; \tilde{z}\big) = \frac{s(8s-9)^2}{27(1-s)}, \qquad
s(z) = S\big(\tf12, \tf1k, \tf2k;z\big) .\end{gather*}

The map between the Schwarz functions for Cases 2(b) and 1(a) of Table~\ref{t:gdh} is
obtained via a~quartic transformation of the ${}_2F_1$ functions~\cite{Vid09}
\begin{gather*}_2F_1(a,a+\tf13;\tf{3a+5}{6}; s) = (1+8s)^{-\frac{3a}{4}}\,{}_2F_1\big(\tf{a}{4},\tf{a}{4}+\tf13;\tf{3a+5}{6}; s\big), \qquad
\tilde{s}=\frac{64s(1-s)^3}{(1+8s)^3} . \end{gather*}
Choosing $a=\tf13-\tf2k$, the above transformation induces a degree 4 map
$(\tf1k, \tf12, \tf13) \mapto4  (\tf1k, \tf3k, \tf13)$. Then
\eqref{map} after applying the Euler transformation to the
${}_2F_1$ function in the denominator, yields the following
\begin{gather*} z(s) = \big[s(1-s)^3\big]^{\frac 1k} \frac{{}_2F_1(\tf13+\tf2k,\tf23+\tf2k;1+\tf1k;s)}
{_2F_1(\tf13-\tf2k, \tf23-\tf2k;1-\tf1k; s)} .\end{gather*}
Applying the quartic transformation to the ${}_2F_1$ functions in the
above expression, one obtains
\begin{gather*} \tilde{z}(\tilde{s}) = (64)^{\frac 1k}z(s) = \tilde{s}^{\frac 1k}
\frac{_2F_1(\tf{1}{12}+\tf{1}{2k},\tf{5}{12}+\tf{1}{2k};1+\tf1k;\tilde{s})}
{_2F_1(\tf{1}{12}-\tf{1}{2k}, \tf{5}{12}-\tf{1}{2k};1-\tf1k; \tilde{s})} .\end{gather*}
Then the substitutions $s \to (1-s)^{-1}$ and $\tilde{s} \to 1-\tilde{s}^{-1}$
yield the degree 4 map
$(\tf12, \tf13, \tf1k) \mapto4 (\tf13, \tf1k, \tf3k)$
between Cases 2(b) and 1(a), given by
\begin{gather*}\tilde{s}(z)=S\big(\tf12, \tf13, \tf1k; \tilde{z}\big) = \frac{(27s^2-36s+8)^2}{64(1-s)},
\qquad s(z) = S\big(\tf13, \tf1k, \tf3k;z\big) .\end{gather*}

The last subcase involving the mapping between Cases 2(c) and 1(a)
of Table~\ref{t:gdh} is derived once again from a quadratic transformation of
hypergeometric functions. It is a degree 2 map
$(\tf12,\tf13,\tf1k) \mapto2 (\tf23,\tf1k,\tf1k)$ given by
\begin{gather*}\tilde{s}=S\big(\tf12,\tf13,\tf1k;\tilde{z}\big) = \frac{(s-2)^2}{4(1-s)}, \qquad
s(z) = S\big(\tf23,\tf1k,\tf1k;z\big) , \qquad \tilde{z} =\big({-}\tf14\big)^{\frac 13} z .\end{gather*}
It is obtained by employing the quadratic transformation
discussed in Case (i) with hypergeometric parameters
$(a,b,c)=(\tf16-\tf1k,\tf56-\tf1k,1-\tf1k)$, and then making the substitutions
$s \to s^{-1}$, $\tilde{s} \to 1-\tilde{s}^{-1}$. Alternatively, one can
start from the quadratic transformation~\cite{Bateman}
\begin{gather*}_2F_1(a,b;2b;s) = (1-s)^{-\frac{a}{2}}
{}_2F_1\big(\tf{a}{2},b-\tf{a}{2};b+\tf12; \tilde{s}\big) , \qquad
\tilde{s}=\frac{s^2}{4(s-1)} , \end{gather*}
with parameters $(a,b,c)=(\tf16-\tf1k,\tf16,\tf13)$, then apply the
map $\tilde{s} \to 1-\tilde{s}$.

{\bf Case (iii).} There are two subcases for Case~3 in Table~\ref{t:gdh}.
For both of these subcases the pull-back
map to Case 1(a) of Table~\ref{t:gdh} is of degree 6, which can be expressed
as a composition of a degree 2 and a degree 3 map. Specif\/ically,
the map between cases 1(a) and 3(a) can be schematically represented as
\begin{gather*}\big(\tf12,\tf13,\tf1k\big) \mapto3 \big(\tf12,\tf1k,\tf2k\big)
\mapto2 \big(\tf1k,\tf1k,\tf4k\big) .\end{gather*}
The degree~2 map above is given by $\theta_1(s)=(2s-1)^2$ as in Case~(i) above.
The degree~3 map given by $\theta_2(s)= s(8s-9)^2/27(1-s)$, is the same as
the one between Cases~1(a) and~2(a) of Table~\ref{t:gdh} and
derived in Case~(ii). Thus, for this
case, $\tilde{s}=\theta_2 \circ \theta_1(s)$.

Similarly, Case 3(b) is mapped to Case 1(a) of Table~\ref{t:gdh} according to
\begin{gather*}\big(\tf12,\tf13,\tf1k\big) \mapto2 \big(\tf13,\tf13,\tf2k\big)
\mapto3 \big(\tf2k,\tf2k,\tf2k\big) .\end{gather*}
The degree 3 map transforming Case 3(b) to the intermediate step
(Case~1(b) of Table~\ref{t:gdh}) follows from a dif\/ferent cubic transformation
from the one in Case~(ii). It is given by~\cite{Vid09}
\begin{gather*} _2F_1\big(a,\tf{a+1}{3}; \tf{2a+2}{3};s\big)=
\big(1+\omega^2s\big)^{-a}\,{}_2F_1\big(\tf{a}{3}, \tf{a+1}{3};\tf{2a+2}{3};\tilde{s}\big),
\qquad \tilde{s}=A\frac{s(s-1)}{(s+\omega)^3} ,\end{gather*}
where $A=3(2\omega+1)$ and $\omega=e^{\frac {2\pi i}{3}}$.
Choosing $a=\half-\tf3k$ and proceeding as in the previous cases,
yields the map $(\tf2k,\tf13,\tf13) \mapto3 (\tf2k,\tf2k,\tf2k)$. Then after the substitution
$\tilde{s} \to \tilde{s}^{-1}$, one obtains
\begin{gather*}\tilde{s}(z)=S\big(\tf13,\tf13,\tf2k; (-A)^{\frac {2}{k}}z\big) =
\frac{(s+\omega)^3}{As(s-1)} := \theta_1(s), \qquad
s(z)=S\big(\tf2k,\tf2k,\tf2k;z\big) .\end{gather*}
The degree 2 map $\theta_2(s)\colon (\tf12,\tf13,\tf1k) \mapto2
 (\tf13,\tf13,\tf2k)$ was already discussed in Case~(i). Hence, the
composition map $\tilde{s}=\theta_2 \circ \theta_1(s)$ gives
the transformation between the Schwarz functions
in Case~3(b) and Case~1(a) of Table~\ref{t:gdh}.

It can be readily verif\/ied that in all the above cases, the rational
map $\tilde{s}=\theta(s)$ is a solution of the dif\/ferential relation
$\phi(\tilde{s}) = \epsilon \phi(s)$ for constant~$\epsilon$,
for the functions $\phi$ listed the last column of Table~\ref{t:gdh}.
An example involving the Cases~1(a) and~1(b) was presented earlier
in Section~\ref{section3.2} (below Table~\ref{t:gdh}). Explicitly, this dif\/ferential
relation reads as
\begin{gather*}{\rm d}x = \frac{{\rm d}\tilde{s}}{\tilde{s}^{1/2}(\tilde{s}-1)^{2/3}} =
 \epsilon \frac{{\rm d}s}{s^{1-\alpha_1}(s-1)^{1-\beta_1}} ,\end{gather*}
where $\tilde{s}(z)$ corresponds to the Schwarz function in Case~1(a)
and $s(z)$ represents the Schwarz function for each of the other
cases in Table~\ref{t:gdh}. It follows from our discussion
in Section~\ref{section3.1} that, geometrically, the above dif\/ferentials
def\/ine the ``f\/lat'' coordinate $x$ on a one-dimensional manifold
$M \subseteq \mathbb{CP}^1$ in which the af\/f\/ine connection
$\eta$ associated with the Chazy XII solution $y(z)$ is trivialized.
Both $s$ and $\tilde{s}=\theta(s)$ are projective invariant coordinates
on $M$. It is interesting to note that solving the above
dif\/ferential relation directly provides an alternative way to derive
the map $\theta(s)$, instead of using the
algebraic transformations of the hypergeometric functions. However,
we do not pursue this direct approach here.
We also remark that the dif\/ferential $\phi(z)dz$ is a primitive of the incomplete
beta integral (see, e.g.,~\cite{Bateman})
\begin{gather*}B(\alpha_1, \beta_1; s) = \int_0^st^{\alpha_1-1}(t-1)^{\beta_1-1}\,{\rm d}t
= \frac{s^{\alpha_1}}{\alpha_1}\,{}_2F_1(\alpha_1, 1-\beta_1;\alpha_1+1;s) ,\end{gather*}
which is also related to the hypergeometric functions. Thus, the rational
map $\tilde{s}= \theta(s)$ corresponds to algebraic transformation of the beta
functions. The derivation of these rational maps directly from
the solutions of \eqref{gch} have not been studied to the best
of our knowledge.

\section{Concluding remarks}\label{section5}
In this article, we have presented a method that converts the
Chazy XII dif\/ferential equation into an algebraic relation.
This method is based on a simple dif\/ferential geometric interpretation
of the Chazy XII equation given by Dubrovin in~\cite{Dubrovin}.
By the way of elucidating this algebraic relationship, we have derived all possible
parametrizations of the Chazy XII solution~$y(z)$ via the Schwarz triangle
functions $S(\alpha, \beta, \gamma;z)$ for $\alpha+\beta+\gamma<1$.
Our method can be extended in a~straightforward way to the
case when $\alpha+\beta+\gamma \geq 1$ but we do not pursue it here.
Furthermore, we show that these parametrizations can be described
in terms of classical hypergeometric functions. Using the hypergeometric
theory, the center and radius of the natural barrier for the Chazy~XII solution
are found explicitly. Finally, the algebraic transformations of
the hypergeometric functions are used to obtain rational maps between
the Schwarz functions corresponding to the Chazy~XII solution.

It is known that the Chazy III equation is equivalent to the Ramanujan
dif\/ferential relations~\eqref{dPQR} for the Eisenstein series~$P$,~$Q$,~$R$
for the modular group ${\rm SL}_2(\mathbb{Z})$. Ramanujan also derived a number of remarkable
identities among the functions $P$, $Q$, $R$ using their representation in terms
of hypergeometric functions. Likewise, we have introduced a Ramanujan-like
triple $(\hat{P}, \hat{Q}, \hat{R})$ which satisfy dif\/ferential relations
that are equivalent to the Chazy~XII equation. Additionally, these
functions also satisfy interesting identities that can be derived from their
hypergeometric representations. We have presented an example at the
end of Section~\ref{section4.1}, but a comprehensive exploration of this line of
investigation is for future study.

The rational maps between the Schwarz functions using the pull-backs
of hypergeometric equations were presented in Section~\ref{section4.2}. However, we
found that these maps can be also obtained by solving certain f\/irst
order dif\/ferential relations among the Schwarz functions. These relations
follow directly from the representation of the Chazy~XII solution~$y(z)$ in terms of the function~$\phi(z)$ in~\eqref{ypara}. We believe
that all rational maps of Schwarz functions satisfy such dif\/ferential relations
although we have not pursued this question further here since
it is beyond the scope of this article. We plan to study this
issue in a future work.

\appendix
\section{Proof of Lemma~\ref{lemma1}}\label{appendixA}
In this appendix we prove Lemma~\ref{lemma1} from Section~\ref{section3.2}.
\begin{proof}
Suppose $J=0$. Then \eqref{u1} and \eqref{u5} lead to 4 cases:
\begin{gather*} (1)\ A = B=0, \qquad (2)\ \alpha_1=\beta_1=0, \qquad (3) \ A=\beta_1=0,
\qquad (4)\ B=\alpha_1=0 .\end{gather*}
When $A=B=0$, the remaining equations yield the following conditions
on $\alpha_1$, $\beta_1$ and $C$
\begin{gather*}(1-2\alpha_1)(1-3\alpha_1)C=(1-2\beta_1)(1-3\beta_1)C\\
\hphantom{(1-2\alpha_1)(1-3\alpha_1)C}{} =
[(1-2\alpha_1)(2-3\beta_1)+ (1-2\beta_1)(2-3\alpha_1)]C=0 .\end{gather*}
If $C=0$, then $A=B=C=0$ implies $\alpha=\alpha_1$, $\beta=\beta_1$,
and $\gamma=\gamma_1$ from the def\/initions of~$A$,~$B$,~$C$ given
in Section~\ref{section3.2}.
Therefore, $\alpha+\beta+\gamma=\alpha_1+\beta_1+\gamma_1=1$, which violates
the condition $\alpha+\beta+\gamma<1$ in~\eqref{cond}. Hence, $C \neq 0$.
Then the f\/irst two conditions above imply
\begin{gather*} \alpha_1 \in \big\{\half, \tf13\big\}, \qquad
\beta_1 \in \big\{\half, \tf13\big\}. \end{gather*}
Except for the case
$\{\alpha_1, \beta_1\} = \{\half, \half\}$, the remaining values of
$\alpha_1$ and $\beta_1$ do not satisfy the third condition given above.
On the other hand, if $\{\alpha_1, \beta_1\} = \{\half, \half\}$, then
$\alpha=\alpha_1=\half$ and $\beta = \beta_1=\half$ because $A=B=0$.
Hence, $\alpha+\beta = 1$, which contradicts \eqref{cond}.

If $\alpha_1=\beta_1=0$, then $\gamma_1=1$. In this case, \eqref{u3} gives
$A+B+C=0$, which implies $\gamma = \gamma_1 =1$, again violating \eqref{cond}.

Cases (3) and (4) are interchangeable by switching $A$ with $B$, and
$\alpha_1$ with $\beta_1$. So it suf\/f\/ices consider only one case, say if
$A=\beta_1=0$. Then \eqref{u2} implies
\begin{gather}
(1-\alpha_1)B + \half C=0
\label{a1}
\end{gather}
Eliminating $B$ from \eqref{u3} by using \eqref{a1} leads to
the condition $(1-2\alpha_1)C=0$. If $\alpha_1=\half$, \eqref{a1}
implies that $B+C=0$ which is equivalent to $\gamma=\gamma_1=\half$
after taking into account $A=0$ and $\beta_1=0$. Hence, $\alpha=\gamma=\half$,
which contradicts $\alpha+\beta+\gamma<1$.

Finally, if $C=0$ in \eqref{a1}, then either $\alpha_1=1$, or $B=0$.
The case $\alpha_1=1$ is impossible because then $A=0$ would imply
$\alpha=\alpha_1=1$, which does not satisfy \eqref{cond}.
If $B=0$, then $A=B=C=0$. This is impossible as argued
in Case (1) above.
\end{proof}

\section{Radius of the orthogonal circle}\label{appendixB}
In this appendix we brief\/ly outline a derivation of
the expression in Section~\ref{section4.1} for the radius~$R$
of the orthogonal circle C, which forms the natural barrier in the
complex plane to the solutions of \eqref{gch}.
\begin{figure}[h]\centering
\includegraphics[scale=0.5]{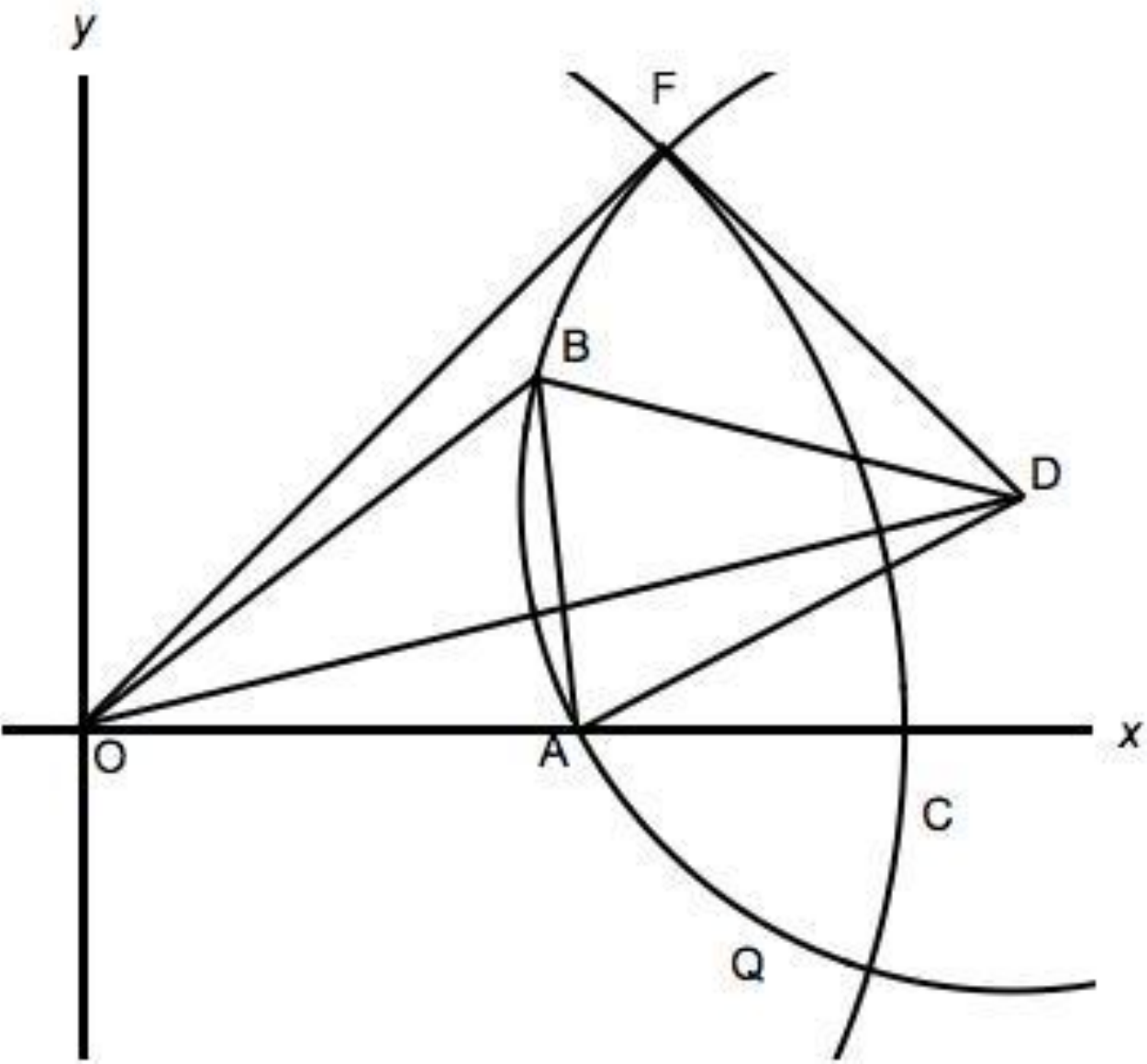}

\caption{The fundamental triangle OAB with one side AB being an arc
of the circle Q with center~D. The orthogonal circle C is centered
at the origin O and has the radius OF=$R$. DF and OF are perpendicular.}\label{fig:R}
\end{figure}
Recall from Section~\ref{section4.1} that \eqref{map} maps the upper-half $s$-plane
onto a circular triangle with two straight edges denoted by~OA and~OB as shown in Fig.~\ref{fig:R}.
The third side of the triangle is formed by the circular arc AB which
is part of the circle Q with center at D, and radius ${\rm AD}={\rm BD}:=r$.
The vertices O, A, B correspond to the image points $z(0)=0$,
$z(1)$ and $z(\infty)$ of the map \eqref{map}, respectively.
The angle $\pi\alpha=\angle {\rm BOA}$, while angles $\pi\beta$ and $\pi\gamma$ are
the angles made by OA and OB respectively, with the circular arc~AB.
The orthogonal circle C is centered at the origin O and its radius
OF:=$R$. Now consider the triangles OAB and DAB whose common side is
given by the chord AB. From elementary considerations, it follows that
$\angle {\rm ADB}=\pi-\pi(\alpha+\beta+\gamma)$, $\angle {\rm
DAB}=\pi(\tf{\alpha+\beta+\gamma}{2})$, and
$\angle {\rm OBA}=\tf{\pi}{2}-\pi(\tf{\alpha+\beta-\gamma}{2})$.
Applying the law of sines to triangles OAB and DAB, and denoting
${\rm OA}=x$, one f\/inds that
\begin{gather*}\mathrm{AB} = \frac{x \sin \pi\alpha}{\cos \pi(\tf{\alpha+\beta-\gamma}{2})},
\qquad \mathrm{BD}:= r = \frac{\mathrm{AB}\sin\pi(\tf{\alpha+\beta+\gamma}{2})}
{\sin \pi(\alpha+\beta+\gamma)} .\end{gather*}
Eliminating AB from above, yields the following expression for the radius
of the circle Q
\begin{gather}
r=\frac{x \sin \pi\alpha}{2\cos \pi(\tf{\alpha+\beta-\gamma}{2})
\cos \pi(\tf{\alpha+\beta+\gamma}{2})} .
\label{r}
\end{gather}
Since the circles Q and C are mutually orthogonal, their radii
OF and DF are perpendicular. Hence from the right triangle ODF
with hypotenuse ${\rm OD}=:d$ it follows that $R^2=d^2-r^2$. Next, considering the
triangle OAD with $\angle {\rm OAD}=\tf{\pi}{2}+\pi\beta$, one obtains
$d^2=x^2+r^2+2rx\sin \pi\beta$. Substituting this into the expression
for $R^2$ and using \eqref{r}, one obtains
\begin{gather}
R^2 = x^2 \frac{\sin \tf{\pi}{2}(1-\alpha+\beta-\gamma)
\sin \tf{\pi}{2}(1-\alpha+\beta+\gamma)}{\sin \tf{\pi}{2}(1-\alpha-\beta+\gamma)
\sin \tf{\pi}{2}(1-\alpha-\beta-\gamma)} ,
\label{R}
\end{gather}
after using some trigonometric identities. From \eqref{map} the line segment
OA$=x$ along the real axis is given by
\begin{gather*}x=z(1)=\frac{_2F_1(a-c+1, b-c+1; 2-c; 1)}{_2F_1(a, b; c; 1)}
=\frac{\Gamma(2-c)\Gamma(c-a)\Gamma(c-b)}{\Gamma(c)\Gamma(1-a)\Gamma(1-b)} ,\end{gather*}
after using $_2F_1(a, b; c; 1)=\Gamma(c)\Gamma(c-a-b)[\Gamma(c-a)\Gamma(c-b)]^{-1}$~\cite{Bateman}.
Substituting the above expression for $x$ into \eqref{R},
using the identity $\Gamma(t)\Gamma(1-t)=\pi\,\mathrm{csc}(\pi t)$, and replacing
$(a, b, c)$ by their values in terms of $(\alpha, \beta, \gamma)$, one f\/inally
obtains the expression for $R^2$ in Section~\ref{section4.1}.

\subsection*{Acknowledgments}
The work of SC was partly supported by NSF grant No. DMS-1410862. The work of OB was supported in part by a CRCW grant from University of Colorado, Colorado Springs. The authors thank Professor Mark Ablowitz for useful discussions, as well as the anonymous referees for their valuable remarks which substantially improved the article.

\pdfbookmark[1]{References}{ref}
\LastPageEnding

\end{document}